\documentclass[11pt]{amsart}
\usepackage{amsmath,amssymb,epsfig,color, amsthm, mathrsfs}
\usepackage{wrapfig,lipsum,floatrow,sidecap}
\usepackage{graphicx}
\usepackage{multirow}
\usepackage{hhline}
\usepackage{url}
\graphicspath{ {images/} }
\allowdisplaybreaks

\newtheorem{theorem}{Theorem}[section]
\newtheorem{conjecture}[theorem]{Conjecture}

\newtheorem{lemma}[theorem]{Lemma}

\newtheorem{proposition}[theorem]{Proposition}
\newtheorem{remark}[theorem]{Remark}

\newcommand{\vanish}[1]{}\parskip=12pt

\newcommand{\0}{\widehat{0}}
\newcommand{\1}{\widehat{1}}

\newcommand{\cupdot}{\mathbin{\mathaccent\cdot\cup}}

\def\p{\prime}
\def\pp{{\prime\prime}}

\def\R{\mathbf{R}}

\def\Z{\mathbb{Z}}
\def\Q{\mathbb{Q}}
\def\K{\mathcal{K}}
\def\L{\mathcal{L}}

\def\O{\Omega}
\def\M{\mathcal{M}}

\def\x{{\bf{x}}}
\def\y{{\bf{y}}}
\def\z{{\bf{z}}}
\def\t{{\bf{t}}}
\def\s{{\bf{s}}}
\def\k{{\bf{k}}}
\def\l{{\bf{l}}}
\def\f{{\bf{f}}}

\def\0{{\bf{0}}}
\def\C{{\mathbb{C}}}
\def\F{\mathcal{F}}

\numberwithin{equation}{section}
\begin{document}
\title{Gabor Functional Multiplier in the Higher Dimensions}
\author{Zhongyan Li $^\sharp$ and Yuanan Diao$^\dagger$ }
\address{$^\sharp$ School of Mathematics and Physics\\
North China Electric Power University, Beijing, 102206, China}
\email{lzhongy@ncepu.edu.cn}
\address{$^\dagger$ Department of Mathematics and Statistics\\
University of North Carolina Charlotte\\
Charlotte, NC 28223}
%\address{$^\sharp$ }
\email{ydiao@uncc.edu}
\thanks{$^\sharp$ is supported by the
National Natural Science Foundation of China (Grant No. 11571107)}
\begin{abstract}
For two given full-rank lattices  $\L=A\Z^d$ and $\K=B\Z^d$ in $\R^d$, where $A$ and $B$ are nonsingular real $d\times d$ matrices, a function $g(\t)\in L^2(\R^d)$ is called a Parseval Gabor frame generator if
$\sum_{\l,\k\in\Z^d}|\langle f, {e^{2\pi i\langle B\k,\t\rangle}}g(\t-A\l)\rangle|^2=\|f\|^2$
holds for any $f(\t)\in L^2(\R^d)$. It is known that Parseval Gabor frame generators exist if and only if $|\det(AB)|\le 1$. A function $h\in L^{\infty}(\R^d)$ is called a functional Gabor frame multiplier if it has the property that $hg$ is a Parseval Gabor frame generator for $L^2(\R^d)$
whenever $g$ is. It is conjectured that an if and only if condition for a function $h\in L^{\infty}(\R^d)$ to be a functional Gabor frame multiplier is that $h$ must be unimodular and $h(\x)\overline{h(\x-(B^T)^{-1}\k)}=h(\x-A\l)\overline{h(\x-A\l-(B^T)^{-1}\k)},\ \forall\ \x\in \R^d$ {\em a.e.} for any $\l,\k\in \Z^d$, $\k\not=\0$. The if part of this conjecture is true and can be proven easily, however the only if part of the conjecture has only been proven in the one dimensional case to this date. In this paper we prove that the only if part of the conjecture holds in the two dimensional case.
\end{abstract}

\maketitle

\section{Introduction}

The characterization of the various frame generator multipliers is a topic that was initially motivated by Dai and Larson on wandering vector multipliers in \cite{DL} and by the WUTAM paper on basic properties of wavelets \cite{WC}. One important application of wavelet multipliers the construction of new wavelets from existing ones. In the one-dimensional as well as in the high-dimensional cases,  wavelet multipliers or Gabor frame wavelet multipliers and multipliers for group frames have been studied extensively and characterized completely. For a representative list of publications on this subject see \cite{GH2003,Han-TAMS,LDDX,LHan, LHan-LAA, WC}.

\medskip
Given two full-rank lattices  $\L=A\Z^d$ and $\K=B\Z^d$ in $\R^d$, where $A$ and $B$ are nonsingular real $d\times d$ matrices. The Gabor family generated by  a function
$g(\t)\in L^2(\R^d)$ is given by
$$G(\L,\K,g)=\{e^{2\pi i\langle B\k,\t\rangle}g(\t-A\l):\,\l,\k\in\Z^d\}.$$
We say that $g(\t)$  is a Gabor frame  generator for $L^{2}(\R^d)$ if $(g, \L,\K)$ is a frame, {\em i.e.},  there exist constants $a_1$, $a_2$ such that $0<a_1\le a_2$ and
$$a_1\|f\|^2\leq \sum_{\l,\k\in\Z^d}|\langle f, {e^{2\pi i\langle B\k,\t\rangle}}g(\t-A\l)\rangle|^2\leq a_2\|f\|^2$$
for any $f(\t)\in L^2(\R^d).$ $g(\t)$ is called a {\em Parseval (or normalized tight) Gabor frame generator} if $a_1=a_2=1$. Czaja provided a complete characterization of a normalized tight Gabor frame function $g$ in \cite{Czaja2000} as stated in the following proposition.

\begin{proposition}\label{prop1}\cite{Czaja2000}
$g$ is a normalized tight Gabor frame function (with $\L=A\Z^d$ and $\K=B\Z^d$) if and only if it satisfies the following two conditions for any $\x\in \R^d$ {\em a.e.} and $\k\not=\0\in \Z^d$:
\begin{eqnarray}
\sum_{\l\in\Z^d}|g(\x-A\l)|^2&=&b, \label{cond1}\\
\sum_{\l\in\Z^d}g(\x-A\l)\overline{g(\x-A\l-(B^T)^{-1}\k)}&=&0,\label{cond2}
\end{eqnarray}
where $b=|\det(B)|$.
\end{proposition}

Furthermore, it has been shown by Han and Wang \cite{HW2001} that a single function Gabor frame generator exists for $L^2(\R^d)$ if and only if $|\det(AB)|\le 1$.

\medskip
A function $h\in L^{\infty}(\R^d)$ is called a {\em functional Gabor frame multiplier} if it has the property that $hg$ is a normalized tight Gabor frame generator whenever $g$ is.
Gu and Han \cite{GH2003} investigated functional Gabor frame multipliers in the one-dimensional case and obtained the following characterization: $h\in L^{\infty}(\R)$ is a functional Gabor frame multiplier for the
full-rank lattices  $\K=a\Z$ and $\L=b\Z$ if and only if it is unimodular and $h(t)\overline{h(t+\frac{1}{b})}$ is $a$-periodic. The generalizations of this result in the higher dimensions (one form of which is stated in Conjecture \ref{conj1} below) may seem plausible, but remain an unsolved problem. This is the main subject of study in this paper. The difficulty here is due to the fact that such characterizations require rather technical/sophisticated treatments of complicated cases that occur in the high dimensional time-frequency lattices. More work on and applications of high dimensional Gabor frames in image signal processing can be found in \cite{Casa, FH1, FH2}.

\medskip
\begin{conjecture}\label{conj1}
Given two full-rank lattices  $\L=A\Z^d$ and $\K=B\Z^d$ in $\R^d$, where $A$ and $B$ are nonsingular real $d\times d$ matrices such that $|\det(AB)|\le 1$, then a function $h\in L^{\infty}(\R^d)$ is a functional Gabor frame multiplier if and only if $h$ is unimodular and $h(\x)\overline{h(\x-(B^T)^{-1}\k)}=h(\x-A\l)\overline{h(\x-A\l-(B^T)^{-1}\k)}$ for any $\x\in \R^d$ {\em(a.e.)}, $\l,\k\in \Z^d$, $\k\not=\0$.
\end{conjecture}

\medskip
Let $h\in L^{\infty}(\R^d)$ such that
$h(\x)\overline{h(\x-(B^T)^{-1}\k)}=h(\x-A\l)\overline{h(\x-A\l-(B^T)^{-1}\k)}$ for any $\x\in \R^d$, $\l,\k\in \Z^d$, $\k\not=\0$ and $|h(\x)|=1$ for any $\x\in \R^d$, and let $g$ be a normalized tight Gabor frame generator. Then we have
\begin{eqnarray*}
\sum_{\l\in\Z^d}|h(\x-A\l)g(\x-A\l)|^2&=&\sum_{\l\in\Z^d}|g(\x-A\l)|^2=b
\end{eqnarray*}
and
\begin{eqnarray*}
&&\sum_{\l\in\Z^d}h(\x-A\l)g(\x-A\l)\overline{h(\x-A\l-(B^T)^{-1}\k)g(\x-A\l-(B^T)^{-1}\k)}\\
&=&\sum_{\l\in\Z^d}h(\x-A\l)\overline{h(\x-A\l-(B^T)^{-1}\k)}g(\x-A\l)\overline{g(\x-A\l-(B^T)^{-1}\k)}\\
&=&
\sum_{\l\in\Z^d}h(\x)\overline{h(\x-(B^T)^{-1}\k)}g(\x-A\l)\overline{g(\x-A\l-(B^T)^{-1}\k)}\\
&=&
h(\x)\overline{h(\x-(B^T)^{-1}\k)}\sum_{\l\in\Z^d}g(\x-A\l)\overline{g(\x-A\l-(B^T)^{-1}\k)}\\
&=&0
\end{eqnarray*}
by equations (\ref{cond1}) and (\ref{cond2}). Thus $hg$ is a normalized tight Gabor frame generator by Proposition \ref{prop1}. This shows that the if part of Conjecture \ref{conj1} is true as stated in the following theorem.

\begin{theorem}\label{Theorem0}
Given two full-rank lattices  $\L=A\Z^d$ and $\K=B\Z^d$ in $\R^d$, where $A$ and $B$ are nonsingular real $d\times d$ matrices such that $|\det(AB)|\le 1$, then a function $h\in L^{\infty}(\R^d)$ is a functional Gabor frame multiplier if $h$ is unimodular and $h(\x)\overline{h(\x-(B^T)^{-1}\k)}=h(\x-A\l)\overline{h(\x-A\l-(B^T)^{-1}\k)}$ for any $\x\in \R^d$, $\l,\k\in \Z^d$, $\k\not=\0$.
\end{theorem}

\medskip
The main result of this paper is the proof that the only if part of Conjecture \ref{conj1} holds when $d=2$.

\begin{theorem}\label{Theorem1}
Let $A$ and $B$ be $2\times 2$ real matrices such that $|\det(AB)|\le 1$, then a function $h$ is a functional Gabor frame multiplier (under the lattices  $\L=A\Z^2$ and $\K=B\Z^2$) only if $h$ is unimodular and $h(\x)\overline{h(\x-(B^T)^{-1}\k)}=h(\x-A\l)\overline{h(\x-A\l-(B^T)^{-1}\k)}$ for any $\l,\k\in \Z^2$, $\k\not=\0$.
\end{theorem}

In the next section, we introduce and prove a few important lemmas in preparation for the proof of Theorem \ref{Theorem1}. This approach allows us to simplify the proof as it involves many different cases. The proof of the main theorem is given in Sections \ref{Sec_Proof} and \ref{Sec_Proof2}: Section \ref{Sec_Proof} covers the case when the eigen values of $(B^TA)^{-1}$ have absolute values at least one and Section \ref{Sec_Proof} covers the case when one of the eigen values of $(B^TA)^{-1}$ has absolute value less than one. In the last section, we discuss several cases where our result can be generalized to higher dimensions and the challenges for the questions that remain unanswered.

\section{Some Simplifications and Lemmas}\label{Sec_Lemma}

In this section, $A$ and $B$ are nonsingular real valued $2\times 2$ matrices such that $|\det(AB)|\le 1$. We consider the normalized tight Gabor frames, namely functions $g\in L^2(\R^2)$ with the property that $\{e^{2\pi i\langle B\k,\x\rangle}g(\x-A\l): \l,\k\in \Z^2\}$ forms a normalized tight frame.

\medskip
Since $(B^TA)^{-1}$ is a real matrix, there exists a real valued $2\times 2$ orthogonal matrix $P$ such that
$P^{-1}(B^TA)^{-1}P=D$ is the real canonical Jordan form of $(B^TA)^{-1}$. Let $\z=(AP)^{-1}\x$, $\L=P^{-1}\Z^2$ and $\K=D\L$, then a normalized tight Gabor frame system $\{g(\x-A\l-(B^T)^{-1}\k): \l,\k\in \Z^2\}$ can be rewritten as
$\{g(AP(\z-\l-\k)): \l\in \L,\ \k\in\K\}=\{\tilde{g}(\z-\l-\k): \l\in \L,\ \k\in\K\}$ by a variable substitution $g(AP(\x))=\tilde{g}(\x)$. The following result is immediate.

\begin{lemma}\label{Lemma2}
If $\tilde{g}(\z)$ is a normalized tight Gabor frame function under the lattices $\L=P^{-1}\Z^2$ and $\K=D\L$, that is, the system $\{e^{2\pi i\langle (D^T)^{-1}\k,\z\rangle}\tilde{g}(\z-\l): \l,\k\in \L\}$ is a normalized tight Gabor frame, then so is $\{e^{2\pi i\langle B\k,\x\rangle} g(\x-A\l): \l,\k\in \Z^2\}$ with $\z=(AP)^{-1}\x$ and $g(\x)=\tilde{g}((AP)^{-1}\x)$. On the other hand, if $\tilde{h}(\z)$ is a functional multiplier such that $\tilde{h}(\z)\overline{\tilde{h}(\z-\k)}=\tilde{h}(\z-\l)\overline{\tilde{h}(\z-\l-\k)}$ for any $\z\in \R^2$, $\l\in \L=P^{-1}\Z^2$, $\k\in\K=D\L$, $\k\not=\0$, then $h(\x)\overline{h(\x-(B^T)^{-1}\k)}=h(\x-A\l)\overline{h(\x-A\l-(B^T)^{-1}\k)}$ for any $\x\in \R^2$, $\l,\k\in \Z^2$, $\k\not=\0$.
\end{lemma}

With Lemma \ref{Lemma2}, we now only need to prove Theorem \ref{Theorem1} for the special case where $A=\begin{bmatrix}1&0\\0&1\end{bmatrix}$ and $(B^T)^{-1}=D$ is a matrix that is of the real canonical Jordan form, since $\L$ can be regarded as $\Z^2$ using its two standard basis element $\vec{e}_1=P^{-1}\begin{bmatrix}1\\0\end{bmatrix}$ and $\vec{e}_2=P^{-1}\begin{bmatrix}0\\1\end{bmatrix}$. More specifically, it suffices for us to prove that if $h$ is a functional multiplier for normalized tight Gabor frame functions under the lattices $\L=\Z^2$ and $\K=D\L$, then $h$ must be unimodular and the following equation must hold
\begin{eqnarray}
h(\x)\overline{h(\x-\k)}&=&h(\x-\l)\overline{h(\x-\l-\k)} \label{special_form}
\end{eqnarray}
for any $\x\in \R^2$,  $\l\in \L$, $\0\not=\k\in \K$. Furthermore, equations (\ref{cond1}) and (\ref{cond2}) become
\begin{eqnarray}
\sum_{\l\in\L}|g(\x-\l)|^2&=&d_0, \label{cond12}\\
\sum_{\l\in\L}g(\x-\l)\overline{g(\x-\l-\k)}&=&0\ \forall\ \0\not=\k\in \K \label{cond22}
\end{eqnarray}
where $d_0=|\det(D^{-1})|$.

Let $\O\subset \R^2$ be a measurable set and $\F$ any given full rank lattice of $\R^2$. We say that $\O$ packs $\R^2$ by $\F$ if $\O\cap (\O+\f)=\emptyset $ for any nontrivial $\f\in \F$. Furthermore, if we also have  $\R^2=\cup_{\f\in \F}(\O+\f)$ then we say that $\O$ tiles $\R^2$ by $\F$. In general, for any two measurable sets $S_1$ and $S_2$ that pack $\R^2$ by $\F$, we say that $S_1$ and $S_2$ are {\em $\F$ equivalent} if for each $\x\in S_1$, $\x=\y+\f$ for some $\y\in S_2$ and $\f\in \F$, and we say that $S_1$ and $S_2$ are {\em $\F$ disjoint} if $S_1\cap (S_2+\f)=\emptyset$ for any nontrivial $\f\in \F$.
By \cite{HW2001}, there exists a measurable set $\O\subset \R^2$ such that $\O$ tiles $\R^2$ by $\L$ and packs $\R^2$ by $\K$. Since the unit square $S=[0,1)\times [0,1)$ tiles $\R^2$ by $\L$ and  the parallelogram $R$ spanned by the two column vectors of $D$ tiles $\R^2$ by $\K$, this means that $\O$ is $\L$ equivalent to $S$ and $\K$ equivalent to a subset of $R$.

\begin{lemma}\label{Lemma4}
If $h$ is a functional Gabor frame multiplier, then $h$ is unimodular.
\end{lemma}

\begin{proof}
Let $\O\subset \R^2$ be a measurable set that tiles $\R^2$ by $\L$ and packs $\R^2$ by $\K$. Since this statement applies to any translation of $\O$, it is obvious that $g(\x)$ defined by $g(\x)=\sqrt{d_0}\chi_{\O^\p}$ is a normalized tight frame, where $\O^\p$ is any translation of $\O$. Since for any $\x\in\R^2$, $\x\in \chi_{\O^\p}$ for some suitably chosen translation of $\O$, and $h\sqrt{d_0}\chi_{\O^\p}$ is a normalized tight Gabor frame, it follows that $|h(\x)|=1$.
\end{proof}

Thus the only thing remained to be proven for Theorem \ref{Theorem1} is that a functional multiplier must satisfy equation (\ref{special_form}).

\medskip
Let $\l_1=(1,0)^T$, $\l_2=(0,1)^T$, $\k_1=D(1,0)^T$ and $\k_2=D(0,1)^T$ be the standard basis for $\L$ and $\K=D\L$. The following lemmas provide us some very useful tools in the proof of Theorem \ref{Theorem1}.

\medskip

\begin{lemma}\label{Lemma5}
Let $h(\x)$ be a unimodular function that satisfies the following conditions for any $\x\in \R^2$:
\begin{eqnarray*}
&&h(\x)\overline{h(\x-\k_1)}=h(\x-\l_1)\overline{h(\x-\l_1-\k_1)},\quad h(\x)\overline{h(\x-\k_1)}=h(\x-\l_2)\overline{h(\x-\l_2-\k_1)},\\
&&h(\x)\overline{h(\x-\k_2)}=h(\x-\l_1)\overline{h(\x-\l_1-\k_2)},\quad
h(\x)\overline{h(\x-\k_2)}=h(\x-\l_2)\overline{h(\x-\l_2-\k_2)},
\end{eqnarray*}
then (\ref{special_form}), hence the statement of Theorem \ref{Theorem1}, holds. Furthermore, in each of the above equations, $-\l_j$ can be replaced by $+\l_j$ and $-\k_j$ can also be replaced by $+\k_j$ ($j=1,2$).
\end{lemma}
\begin{proof} If $h(\x)\overline{h(\x-\k_j)}=h(\x-\l_1)\overline{h(\x-\l_1-\k_j)}$ for any $\x$ ($j=1,2$), then replacing $\x$ by $\x-\l_1$ yields $h(\x)\overline{h(\x-\k_j)}=h(\x+\l_1)\overline{h(\x+\l_1-\k_j)}$. Similarly, if $h(\x)\overline{h(\x-\k_j)}=h(\x+\l_1)\overline{h(\x+\l_1-\k_j)}$, then replacing $\x$ by $\x+\l_1$ yields $h(\x)\overline{h(\x-\k_j)}=h(\x-\l_1)\overline{h(\x-\l_1-\k_j)}$. Similar argument can be applied to $\x$ and $\x\pm \k_j$. This proves the second part of the Lemma. So the given condition of the Lemma assures that $h(\x)\overline{h(\x\pm \k_j)}=h(\x\pm \l_i)\overline{h(\x\pm \l_i\pm\k_j)}$ for any $i,j\in\{1,2\}$ and any $\x\in \R^2$. From here it is easy to prove that $h(\x)\overline{h(\x-\k)}=h(\x-\l)\overline{h(\x-\l-\k)}$ for any $\l\in \L$, $\k\in \K$ and $\x\in \R^2$ by induction since
for any $\l\in \L$ and $\k\in \K$, we have $\l=m_1\l_1+m_2\l_2$ and $\k=n_1\k_1+n_2\k_2$ for some integers $m_1$, $m_2$, $n_1$ and $n_2$.
\end{proof}

\medskip
\begin{lemma}\label{Lemma6}
Let $h$ be a functional Gabor frame multiplier and let $\{\l_i, \k_j\}$ be a given pair with $1\le i,j\le 2$,  then $h(\x)\overline{h(\x\pm \k_j)}=h(\x\pm \l_i)\overline{h(\x\pm \l_i\pm \k_j)}$ for any $\x\in \R^2$ if one of the following conditions holds:\\
\noindent
(i) There exist disjoint and measurable sets $E_1$, $E_2$, $E_3$, $E_4$ and $E_5$, with $E_2$ being either a rectangle or a parallelogram, such that $E_3=E_2+\k_j$, $E_4=E_2+\l_i$, $E_5=E_3+\l_i=E_4+\k_j$, and $E_1\cup E_2\cup E_3$ tiles $\R^2$ by $\L$ while $E_1\cup E_2\cup E_4$ packs $\R^2$ by $\K$;\\
\noindent
(ii) $\k_j\in \L$, $\k_j\not=\pm \l_i$ and there exist disjoint and measurable sets $E_1$, $E_2$, $E_3$, $E_4$ and $E_5$, with $E_2$ being either a rectangle or a parallelogram, such that $E_3=E_2+\k_j$, $E_4=E_2+\l_i$, and $E_5=E_3+\l_i=E_4+\k_j$, $E_1\cup E_2$ tiles $\R^2$ by $\L$ and $E_1\cup E_2\cup E_4$ packs $\R^2$ by $\K$;\\
\noindent
(iii) $\k_j=\pm \l_i$.
\end{lemma}

\begin{proof}
(i) First we claim that the function $g(\x)$ defined by
$$
g(\x)=\sqrt{d_0}\chi_{_{E_1}}+\sqrt{\frac{d_0}{2}}\left(-\chi_{_{E_2}}+\chi_{_{E_3}}+\chi_{_{E_4}}+\chi_{_{E_5}}\right)
$$
is a normalized tight Gabor frame function. To see this, we only need to verify equations (\ref{cond12}) and (\ref{cond22}) for $\x\in E_1\cup E_2\cup E_3$. If $\x\in E_1$, then for any $2\le i\le 5$, $\x-\l\not\in E_i$ and $\x-\k\not\in E_i$ for any $\l\in \L$, $\k\in \K$. Hence (\ref{cond12}) and (\ref{cond22}) hold for $\x\in E_1$. If $\x\in E_2$, then $g(\x-\l)=0$ unless $\l=\0$ or $\l=-\l_i$. So $\sum_{\l\in\L}|g(\x-\l)|^2=d/2+d/2=d$. Furthermore, for $\l=\0$ and $\k\not=\0$, $g(\x-\l-\k)=g(\x-\k)\not=0$ only if $\k=-\k_j$ since $E_1\cup E_2\cup E_4$ packs $\R^2$ by $\K$, $E_3=E_2+\k_j$ and $E_5=E_4+\k_j$. Similarly, for $\l=-\l_i$ and $\k\not=\0$, $g(\x-\l-\k)=g(\x+\l_i-\k)\not=0$ only if $\k=\0$ or $\k=-\k_j$. Thus (\ref{cond22}) holds trivially for any nontrivial $\k$ such that $\k\not=-\k_j$. For $\k=-\k_j$, we have
\begin{eqnarray*}
&&\sum_{\l\in\L}g(\x-\l)\overline{g(\x-\l+\k_j)}\\
&=&g(\x)\overline{g(\x+\k_j)}+g(\x+\l_i)\overline{g(\x+\l_i+\k_j)}\\
&=&
-\sqrt{\frac{d_0}{2}}\cdot \sqrt{\frac{d_0}{2}}+\sqrt{\frac{d_0}{2}}\cdot \sqrt{\frac{d_0}{2}}=0.
\end{eqnarray*}
Similarly, (\ref{cond12}) and (\ref{cond22}) hold for any $\x\in E_3$ and any $\k\not=\0$. The details are left to the reader.

\medskip
Now, let $h$ be a functional multiplier. Since $g(\x)$ is a normalized tight Gabor frame function, so is $(hg)(\x)$. Furthermore, since for any $\y\in \R^2$, the sets $E^\p_q=E_q+\y$ ($1\le q\le 5$) also satisfy the conditions in Lemma \ref{Lemma6}(i), hence we only need to consider the case $\x\in E_2$ and (\ref{cond22}) with $g$ replaced by $hg$:
\begin{eqnarray*}
&&\sum_{\l\in\L}(hg)(\x-\l)\overline{(hg)(\x-\l+\k_j)}\\
&=&(hg)(\x)\overline{(hg)(\x+\k_j)}+(hg)(\x+\l_i)\overline{(hg)(\x+\l_i+\k_j)}\\
&=&
-\frac{d_0}{2}h(\x)\overline{h(\x+\k_j)}+\frac{d_0}{2}h(\x+\l_i)\overline{h(\x+\l_i+\k_j)}=0.
\end{eqnarray*}
It follows that $h(\x)\overline{h(\x+\k_j)}=h(\x+\l_i)\overline{h(\x+\l_i+\k_j)}$ for any $\x\in \R^2$. The other cases now follow from the second part of Lemma \ref{Lemma5} and its proof.

\medskip
(ii) We claim that the function $g(\x)$ defined by:
$$
g(\x)=\sqrt{d_0}\chi_{_{E_1}}+\frac{\sqrt{d_0}}{2}\left(-\chi_{_{E_2}}+\chi_{_{E_3}}+\chi_{_{E_4}}+\chi_{_{E_5}}\right).
$$
is a normalized tight Gabor frame function and proceed to verify that (\ref{cond12}) and (\ref{cond22}) hold for any $\x\in E_1\cup E_2$.
Notice that if $\x\in E_1$, then (\ref{cond12}) and (\ref{cond22}) hold trivially since $(E_1+\l)\cap E_m=(E_1+\k)\cap E_m=\emptyset$ for any nontrivial $\l\in \L$ and $\k\in \K$ and $2\le m\le 5$ by the given condition. On the other hand, if $\x\in E_2$, then the summation in (\ref{cond12}) contains exactly four terms corresponding to $\l=\0,-\l_i,-\k_j$ and $\l=-\l_i-\k_j$ (keep in mind that $\k_j\in \L$ in this case) and each term equals $d_0/4$, hence (\ref{cond12}) holds. This also means that we only need to consider the four terms in the summation of (\ref{cond22}) corresponding to these $\l$ vectors. Keep in mind that $\k\not=\0$ in (\ref{cond22}) and $\x\in E_2$. Thus for $\l=\0$, $g(\x-\k)\not=0$ only if $\k=-\k_j$ since $E_3=E_2+\k_j$, $E_1\cup E_2\cup E_4$ packs $\R^2$ by $\K$ and $E_5=E_4+\k_j$; Similarly, for $\l=-\l_i$, $\x-\l=\x+\l_i\in E_4$, hence $g(\x+\l_i-\k)\not=0$ only if $\k=-\k_j$ (since $E_2+\l_i=E_4$ and $E_5=E_4+\k_j$); for $\l=-\k_j$, $\x-\l=\x+\k_j\in E_3$ hence $g(\x+\l_i-\k)\not=0$ only if $\k=\k_j$; for $\l=-\l_i-\k_j$, $\x-\l=\x+\l_i+\k_j\in E_5$ hence $g(\x+\l_i-\k)\not=0$ also only if $\k=\k_j$. Combining the above, we see that (\ref{cond22}) holds trivially if $\k\not=\pm \k_j$. For $\k=-\k_j$, we have
\begin{eqnarray*}
&&\sum_{\l\in\L}g(\x-\l)\overline{g(\x-\l+\k_j)}\\
&=&g(\x)\overline{g(\x+\k_j)}+g(\x+\l_i)\overline{g(\x+\l_i+\k_j)}\\
&=&
-\sqrt{\frac{d_0}{2}}\cdot \sqrt{\frac{d_0}{2}}+\sqrt{\frac{d_0}{2}}\cdot \sqrt{\frac{d_0}{2}}=0.
\end{eqnarray*}
Similarly, for
$\k=\k_j$, we have
\begin{eqnarray*}
&&\sum_{\l\in\L}g(\x-\l)\overline{g(\x-\l+\k_j)}\\
&=&g(\x+\k_j)\overline{g(\x+\k_j-\k_j)}+g(\x+\l_i+\k_j)\overline{g(\x+\l_i+\k_j-\k_j)}\\
&=&
-\sqrt{\frac{d_0}{2}}\cdot \sqrt{\frac{d_0}{2}}+\sqrt{\frac{d_0}{2}}\cdot \sqrt{\frac{d_0}{2}}=0.
\end{eqnarray*}
The rest of the proof is similar to (i) above and is left to the reader.

\medskip
(iii) WLOG let us assume that $\l_1=\k_1$ since the other cases are analogous. Let $\O$ be a measurable set that tiles $\R^2$ by $\L$ and packs $\R^2$ by $\K$ (the existence of such a set is proven in \cite{HW2001}). Following \cite{GH2003}, let $U$ be the bilateral-shift unitary operator on $\ell^2(\Z)$ and let $\eta=\{\eta_n\}$ be any element in $\ell^2(\Z)$ such that $\{U^k\eta: k\in \Z\}$ is an orthonormal basis of $\ell^2(\Z)$.
Define
$$
g(\x)=\left\{
\begin{array}{ll}
\sqrt{d_0}\eta_n &{\rm if}\ \x+n\l_1\in \O;\\
0 &{\rm otherwise}.
\end{array}
\right.
$$
Notice that we only need to verify (\ref{cond12}) and (\ref{cond22}) for any $\x\in \O$ since $\Omega$ tiles $\R^2$ by $\L$. Thus for any $\x\in \O$, $g(\x-\l)=0$ unless $\l=n\l_1$ for some integer $n$. It follows that
$$
\sum_{\l\in\L}|g(\x-\l)|^2=\sum_{n\in \Z}|g(\x-n\l_1)|^2=d_0\sum_{n\in \Z}|\eta_n|^2=d_0.
$$
On the other hand, $(\O+n\l_1)\cap(\O+\k)=(\O+n\k_1)\cap(\O+\k)=\emptyset$ for any $n\in \Z$ unless $\k=n\l_1=n\k_1$ since
$\O$ packs $\R^2$ by $\K$, thus
$\sum_{\l\in\L}g(\x-\l)\overline{g(\x-\l-\k)}=0$ if $\k\not=n^\p\l_1$ for any $n^\p$. Finally, if $\k=n^\p\l_1$ for some $n^\p\not=0$, then
$$
\sum_{\l\in\L}g(\x-\l)\overline{g(\x-\l-\k)}=\sum_{n\in\Z}g(\x-n\l_1)\overline{g(\x-n\l_1-n^\p\l_1)}=\sum_{n\in\Z}\eta_n\overline{\eta_{n+n^\p}}=d_0\delta_{0,n^\p}=0.
$$
This proves that $g$ is a normalized tight Gabor frame function. Let $h$ be a functional multiplier and apply (\ref{cond22}) to $hg$, for $\k=n^\p\k_1=n^\p\l_1$ and any $\x\in \O$, define $\alpha_n=h(\x-n\l_1)$, then we have
\begin{eqnarray*}
&&\sum_{\l\in\L}(hg)(\x-\l)\overline{(hg)(\x-\l-\k)}\\
&=&\sum_{n\in\Z}(hg)(\x-n\l_1)\overline{(hg)(\x-n\l_1-n^\p\l_1)}\\
&=&\sum_{n\in\Z}\alpha_n\eta_n\overline{\alpha_{n+n^\p}\eta_{n+n^\p}}\\
&=&\delta_{0,n^\p}=0
\end{eqnarray*}
for any $n^\p\not=0$.
Thus $M_{\alpha}$ with $\alpha=\{\alpha_n\}$ is a wandering vector multiplier for the group $\{U^n:\ n\in \Z\}$. By \cite{GH2003} we have
$$
\alpha_n=h(\x-n\l_1)=\lambda e^{2\pi i n\tau}
$$
for some constants $\lambda$ (with $|\lambda|=1$) and $\tau$ that depend only on $\x$. Thus
$$
h(\x)\overline{h(\x-\k_1)}=h(\x)\overline{h(\x-\l_1)}=\lambda\cdot \overline{\lambda}e^{-2\pi i \tau}=e^{-2\pi i \tau}
$$
and
$$
h(\x-\l_1)\overline{h(\x-\l_1-\k_1)}=\lambda e^{2\pi i \tau}\cdot \overline{\lambda}e^{-4\pi i \tau}=e^{-2\pi i \tau}.
$$
That is, $ h(\x)\overline{h(\x-\k_1)}= h(\x-\l_1)\overline{h(\x-\l_1-\k_1)}$ as desired.
\end{proof}

\medskip
The following lemma is a slight variation to cases (i) and (ii) of Lemma \ref{Lemma6}, which will also be needed later in the proof of our main result.

\medskip
\begin{lemma}\label{Lemma7}
Let $h$ be a functional Gabor frame multiplier and let $\{\l_i, \k_j\}$ be a given pair with $1\le i,j\le 2$,  then $h(\x)\overline{h(\x\pm \k_j)}=h(\x\pm \l_i)\overline{h(\x\pm \l_i\pm \k_j)}$ for any $\x\in \R^2$ if $h(\x)\overline{h(\x\pm \k_j)}=h(\x\pm \l_{i^\p})\overline{h(\x\pm \l_{i^\p}\pm \k_j)}$ for any $\x\in \R^2$ where $1\le {i^\p}\le 2$ and $i+{i^\p}\equiv 1$ mod(2), and one of the following two conditions hold:\\
\noindent
(i) there exist disjoint and measurable sets $E_1$, $E_2$, $E_3$, $E_4$ and $E_5$, with $E_2$ being either a rectangle or a parallelogram, such that $E_3=E_2+\k_j$, $E_4=E_2+\l_i+\l_{i^\p}$, $E_5=E_3+\l_i+\l_{i^\p}$, and $E_1\cup E_2\cup E_3$ tiles $\R^2$ by $\L$ while $E_1\cup E_2\cup E_4$ packs $\R^2$ by $\K$;\\
\noindent
(ii) $\k_j\in \L$, $\k_j\not=\pm \l_i$ and there exist disjoint and measurable sets $E_1$, $E_2$, $E_3$, $E_4$ and $E_5$, with $E_2$ being either a rectangle or a parallelogram, such that $E_3=E_2+\k_j$, $E_4=E_2+\l_{i^\p}+\l_i$, and $E_5=E_3+\l_{i^\p}+\l_i(=E_4+\k_j)$, $E_1\cup E_2$ tiles $\R^2$ by $\L$ and $E_1\cup E_2\cup E_4$ packs $\R^2$ by $\K$.
\end{lemma}

\begin{proof}
(i) Similar to the proof of Lemma \ref{Lemma6}(i), using the fact that the function
$$
g(\x)=\sqrt{d_0}\chi_{_{E_1}}+\sqrt{\frac{d_0}{2}}\left(-\chi_{_{E_2}}+\chi_{_{E_3}}+\chi_{_{E_4}}+\chi_{_{E_5}}\right)
$$
is a normalized tight Gabor frame function, we derive  $h(\x)\overline{h(\x+\k_j)}=h(\x+\l_{i^\p}+\l_i)\overline{h(\x+\l_{i^\p}+\l_i+\k_j)}$ for any $\x\in \R^2$. Substituting $\x+\l_{i^\p}$ by $\x$ on both sides of this, we obtain $h(\x-\l_{i^\p})\overline{h(\x-\l_{i^\p}+\k_j)}=h(\x+\l_i)\overline{h(\x+\l_i+\k_j)}$.
Since $h(\x)\overline{h(\x\pm \k_j)}=h(\x\pm \l_{i^\p})\overline{h(\x\pm \l_{i^\p}\pm \k_j)}$ for any $\x\in \R^2$, this leads to $h(\x)\overline{h(\x+ \k_j)}=h(\x+\l_i)\overline{h(\x+\l_i+\k_j)}$ for any $\x\in \R^2$, hence $h(\x)\overline{h(\x\pm \k_j)}=h(\x\pm \l_i)\overline{h(\x\pm \l_i\pm \k_j)}$ for any $\x\in \R^2$ by Lemma \ref{Lemma5}.

\medskip
(ii) The proof is similar to the above the case (ii) of Lemma \ref{Lemma6} using the function
$$
g(\x)=\sqrt{d_0}\chi_{_{E_1}}+\frac{\sqrt{d_0}}{2}\left(-\chi_{_{E_2}}+\chi_{_{E_3}}+\chi_{_{E_4}}+\chi_{_{E_5}}\right).
$$
At the end we also reach the conclusion $h(\x)\overline{h(\x+\k_j)}=h(\x+\l_{i^\p}+\l_i)\overline{h(\x+\l_{i^\p}+\l_i+\k_j)}$ for any $\x\in \R^2$, which leads to $h(\x)\overline{h(\x+ \k_j)}=h(\x+\l_i)\overline{h(\x+\l_i+\k_j)}$ for any $\x\in \R^2$,  hence $h(\x)\overline{h(\x\pm \k_j)}=h(\x\pm \l_i)\overline{h(\x\pm \l_i\pm \k_j)}$ for any $\x\in \R^2$ by Lemma \ref{Lemma5}.
\end{proof}

\medskip
\begin{remark}\label{Remark7}{\em
From the proofs of Lemmas \ref{Lemma5}, \ref{Lemma6} and \ref{Lemma7} it is easy to see that the results of these lemmas hold if $\l_1$, $\l_2$, $\k_1$ and $\k_2$ are replaced by any $\l_1^\p$, $\l_2^\p\in \L$, $\k_1^\p$, $\k_2^\p\in\K$ so long as $\{\l_1^\p,\l_2^\p\}$ is a basis for $\L$ and $\{\k_1^\p,\k_2^\p\}$  is a basis for $\K$ (with $\L$ and $\K$  treated as modules over $\Z$).}
\end{remark}

\medskip
\section{The tiling of $\R^2$ by $\L$ and $\K$}\label{Tiling}

In order to apply the results from the last section such as Lemmas \ref{Lemma6} and \ref{Lemma7}, we need to know how to construct the various sets satisfying the conditions called for by these lemmas. For this we will rely on Theorem 1.2 of \cite{HW2001} and the results leading to the proof of it. We shall state these results in a slightly modified way to match the setting and the terminology we have adopted in this paper.

Define $\M=\{\l+\k:\ \l\in \L, \k\in\K\}$, $\M_1=\{m\l_1+n\k_1:\ m, n\in \Z\}$ and $\M_2=\{m\l_2+n\k_2:\ m, n\in \Z\}$.
We have the following lemmas, which can be obtained by modifying Lemma 2.2 and its proof, as well as the proof of Theorem 1.2 in \cite{HW2001}. Due to the scope of this paper, we shall omit their proofs and refer our reader to  \cite{HW2001}.

\begin{lemma}\label{newlemma1}\cite{HW2001}
If $\M$ is dense in $\R^2$, then for any $S_0=\cupdot_{1\le j\le p}C_j$ and $R_0=\cupdot_{1\le j\le p}C_j^\p$, where the $C_j$, $C_j^\p$'s are polytopes in $\R^2$, such that $\mu(C_j)= \mu(C_j^\p)$ for each $j$ (where $\mu$ stands for the Lebesgue measure), $S_0$ tiles $\R^2$ by $\L$ and $R_0$ packs $\R^2$ by $\K$, there exists a measurable set $C_j^\pp$  for each $j$ such that $C_j^\pp$ is $\L$ equivalent to $C_j$ and $\K$ equivalent to $C_j^\p$. Consequently, $\O=\cupdot_{1\le j\le p}C_j^\pp$ tiles $\R^2$ by $\L$ and packs $\R^2$ by $\K$. Furthermore, in the case $C_j= C_j^\p$, we may simply choose $C_j^\pp=C_j=C_j^\p$.
\end{lemma}

\begin{lemma}\label{newlemma2}\cite{HW2001}
If $\k_1=\pm \frac{m_1}{n_1}\l_1$, $\k_2=\pm\frac{m_2}{n_2}\l_2$ with $0<m_1<n_1$, $m_1$, $n_1$ coprime, and $m_2>n_2$, $m_2$, $n_2$ coprime, $\frac{m_1m_2}{n_1n_2}\ge 1$, then for any
 $S_0=\cupdot_{1\le j\le p}C_j$ and $R_0=\cupdot_{1\le j\le p}C_j^\p$, where the $C_j$, $C_j^\p$'s are rectangles of the form $C_j=[\frac{a_j}{n_1},\frac{a_j+1}{n_1})\times [\frac{b_j}{n_2},\frac{b_j+1}{n_2})$ and $C_j^\p=[\frac{a^\p_j}{n_1},\frac{a^\p_j+1}{n_1})\times [\frac{b^\p_j}{n_2},\frac{b^\p_j+1}{n_2})$ such that $S_0$ tiles $\R^2$ by $\L$ and $R_0$ packs $\R^2$ by $\K$, then for each $j$ there exists a rectangle $C_j^\pp$ such that $C_j^\pp$ is $\L$ equivalent to $C_j$ and $\K$ equivalent to $C_j^\p$. Consequently, $\O=\cupdot_{1\le j\le p}C_j^\pp$ tiles $\R^2$ by $\L$ and packs $\R^2$ by $\K$. Notice that in this case it is necessary that $n_1n_2=p$. Furthermore, we can choose $C_j^\pp=C_j=C_j^\p$ if $C_j=C_j^\p$.
\end{lemma}

\begin{lemma}\label{newlemma3}\cite{HW2001}
If $\k_1=\pm\frac{m_1}{n_1}\l_1$, $\k_2=r_2\l_2$ with $m_1$, $n_1$ coprime, and $r_2$ irrational (in this case $\M_2$ is a dense subset of $\{r\l_2:\ r\in\R\}$), then for any
 $S_0=\cupdot_{1\le j\le p}C_j$ and $R_0=\cupdot_{1\le j\le p}C_j^\p$, where the $C_j$, $C_j^\p$'s are rectangles of the form $C_j=[\frac{a_j}{n_1},\frac{a_j+1}{n_1})\times [{b_j},{c_j})$ and $C_j^\p=[\frac{a^\p_j}{n_1},\frac{a^\p_j+1}{n_1})\times [{b^\p_j},{c^\p_j})$ such that $c_j-b_j=c_j^\p-b_j^\p$,  $S_0$ tiles $\R^2$ by $\L$ and $R_0$ packs $\R^2$ by $\K$, then for each $j$ there exists a measurable set $C_j^\pp$ such that $C_j^\pp$ is $\L$ equivalent to $C_j$ and $\K$ equivalent to $C_j^\p$. Consequently, $\O=\cupdot_{1\le j\le p}C_j^\pp$ tiles $\R^2$ by $\L$ and packs $\R^2$ by $\K$. Again we can choose $C_j^\pp=C_j=C_j^\p$ if $C_j=C_j^\p$. A similar statement holds if $\k_1=r_1\l_1$, $\k_2=\pm\frac{m_2}{n_2}\l_2$ with $r_1$ irrational, and $m_2$, $n_2$ coprime.
\end{lemma}

\medskip
\begin{remark}\label{Remark2.7}{\em
The above lemmas, combined with Lemmas \ref{Lemma6} and  \ref{Lemma7}, can greatly simplify our proof of the main theorem. For example, if $\O$ can be constructed in such a way that it contains a rectangle (or parallelogram) $E_2$ and a measurable set $E_0$ such that $E_0$ is disjoint from $E_2$ and  is $\L$ equivalent to $E_2+\k_j$ and $\K$ equivalent to $E_2+\l_i$ ($1\le i\le 2$ and $1\le j\le 2$), then Lemma \ref{Lemma6}(i) holds for the pair $(\l_i,\k_j)$ with $E_1=\O\setminus (E_2\cup E_0)$, $E_2$, $E_3=E_2+\k_j$, $E_4=E_2+\l_i$ and $E_5=E_2+\l_i+\k_j$.
}
\end{remark}

We now proceed to prove  (\ref{special_form}) with $\L=\Z^2$ and $\K=D\L$. Due to the many different cases of the matrix $D$, we will divide our proof into three parts. In the next part, we will consider the case when the eigen values of $D$ are real and have absolute values at least one.  In Section \ref{Sec_Proof2} we will consider the case when one of the eigen values of $D$ has absolute value less than one. The case when the eigen values of $D$ are complex is last considered in Section \ref{Proof3}.
\section{The proof of the main result: Part 1}\label{Sec_Proof}

When the eigen values of $D$ are real and have absolute values at least one, it can be divided into the following types:

\noindent
Type I: $D=\begin{bmatrix}\pm1&0\\0&r\end{bmatrix}$ or $D=\begin{bmatrix}r&0\\0&\pm 1\end{bmatrix}$ with $|r|>1$;\\
Type II: $D=\begin{bmatrix}r_1&0\\0&r_2\end{bmatrix}$ with $|r_1|>1$ and $|r_2|>1$;\\
%(Ic) $D=\begin{bmatrix}r_1&0\\0&r_2\end{bmatrix}$ with $|r_1r_2|\ge 1$ but either $|r_1|<1$ or $|r_2|<1$;\\
Type III: $D=\begin{bmatrix}\pm 1&0\\0&\pm1\end{bmatrix}$;\\
Type IV: $D=\begin{bmatrix}r&1\\0&r\end{bmatrix}$ with $|r|>1$;\\
Type V: $D=\begin{bmatrix}1&1\\0&1\end{bmatrix}$ or $D=\begin{bmatrix}-1&1\\0&-1\end{bmatrix}$.

\medskip
We shall prove these cases one by one.
\subsection{Type I Case} WLOG we may assume that $D=\begin{bmatrix}1&0\\0&r\end{bmatrix}$ with $r>1$. Since $\k_1=D(1,0)^T=(1,0)^T=\l_1$, $ h(\x)\overline{h(\x-\k_1)}= h(\x-\l_1)\overline{h(\x-\l_1-\k_1)}$ by Lemma \ref{Lemma6}(iii). Notice that $S=[0,1)\times [0,1)$ tiles $\R^2$ by $\L$ and packs $\R^2$ by $\K$ so we can choose $\O=S$. Let $\delta>0$ be small enough so that $\delta<\min\{r-1,1/2\}$ and let $E_1=[0,1)\times [\delta,1)$, $E_2=[0,1)\times [0,\delta)$, $E_3=E_2+\k_1$,
$E_4=E_2+\l_2$, $E_5=E_3+\l_2$. One can verify that $E_1$ through $E_5$ satisfy condition (ii) of Lemma \ref{Lemma6}, hence $h(\x)\overline{h(\x-\k_1)}= h(\x-\l_2)\overline{h(\x-\l_2-\k_1)}$. In order to apply Lemma
\ref{Lemma5}, it remains for us to prove that $h(\x)\overline{h(\x-\k_2)}=h(\x-\l_1)\overline{h(\x-\l_1-\k_2)}$ and $h(\x)\overline{h(\x-\k_2)}=h(\x-\l_2)\overline{h(\x-\l_2-\k_2)}$.

\medskip
Let $q=\lfloor r\rfloor$ so that $r=q+r_0$ with $0\le r_0<1$. There are two possibilities for us to consider: (i) $0<r_0<1$ and (ii) $r_0=0$. In the latter case $r=q\ge 2$ is an integer and $\k_2\in \L$.

\noindent
(i) $0<r_0< 1$. Let $\delta>0$ be small enough so that $\delta<\min\{r_0,1-r_0\}$. Let $E_0=[0,1)\times [r_0,r_0+\delta)$, $E_2=[0,1)\times [0,\delta)$ and $E_1=[1,0)\times [0,1)-(E_0\cup E_2)$, then $E_0$, $E_1$, $E_2$ are disjoint and $E_0\cup E_1\cup E_2=[0,1)\times [0,1)$ tiles $\R^2$ by $\L$ and packs $\R^2$ by $\K$. Let
$E_3=E_2+\k_2$,
$E_4=E_2+\l_2$ and
$E_5=E_3+\l_2$.
Since $E_3-q\l_2=E_0$ as one can easily check, $E_1$, $E_2$, $E_3$ are disjoint and $E_1\cup E_2\cup E_3$ tiles $\R^2$ by $\L$ and $E_1\cup E_2\cup E_4$ packs $\R^2$ by $\K$. Thus the sets $E_1$ through $E_5$ satisfy condition (i) of Lemma \ref{Lemma6} for the pair $\{\l_2,\k_2\}$, hence $h(\x)\overline{h(\x+\k_2)}=h(\x+\l_2)\overline{h(\x+\l_2+\k_2)}$ for any $\x\in\R^2$, $h(\x)\overline{h(\x+\k_2)}=h(\x\pm \l_2)\overline{h(\x\pm\l_2+\k_2)}$ for any $\x\in\R^2$ by Lemma \ref{Lemma5}.

\begin{figure}[ht!]
\begin{center}
\includegraphics[height=2.5in,width=2.0in]{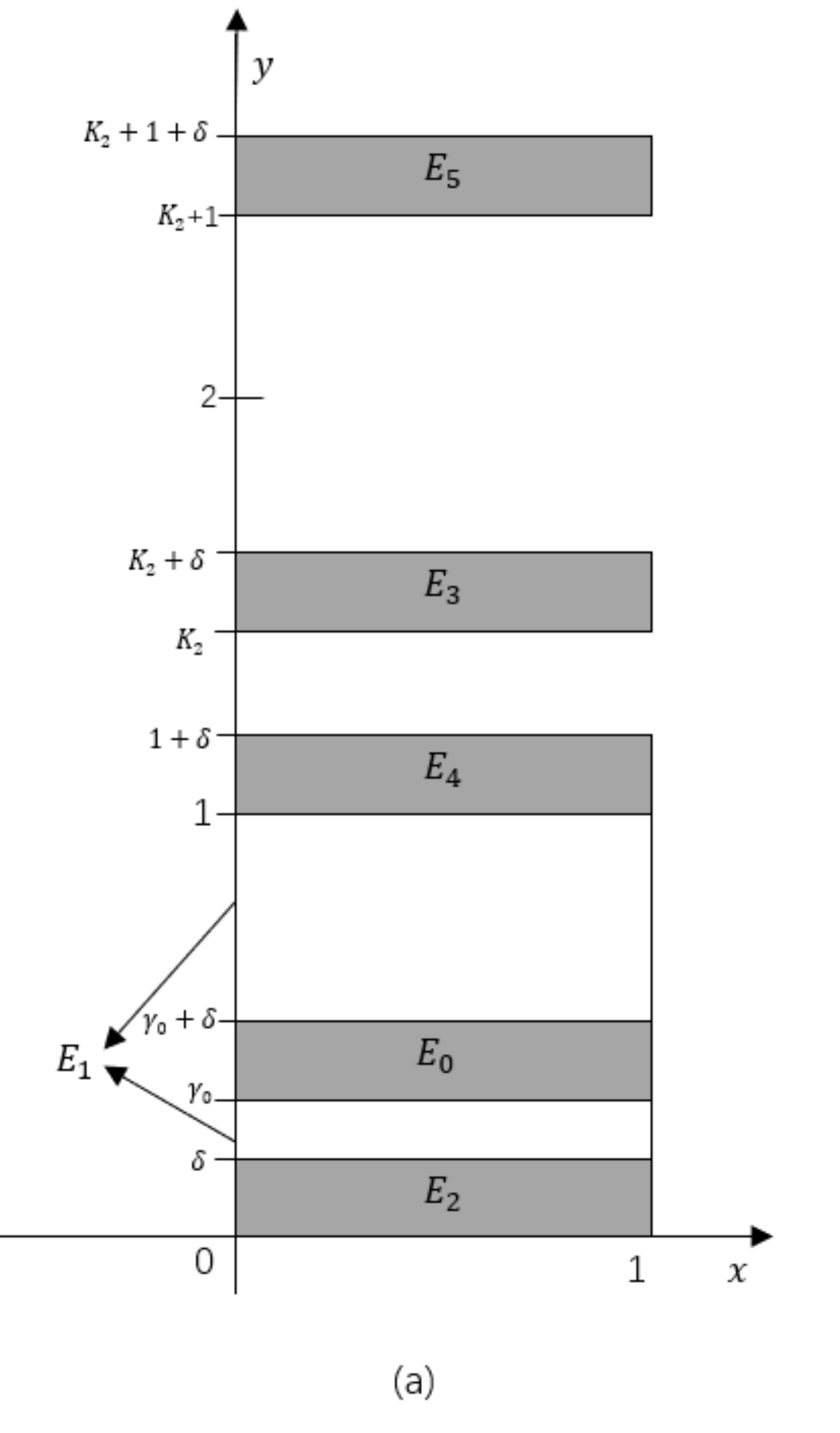}\hspace{0.1in}
\includegraphics[height=2.5in,width=2.8in]{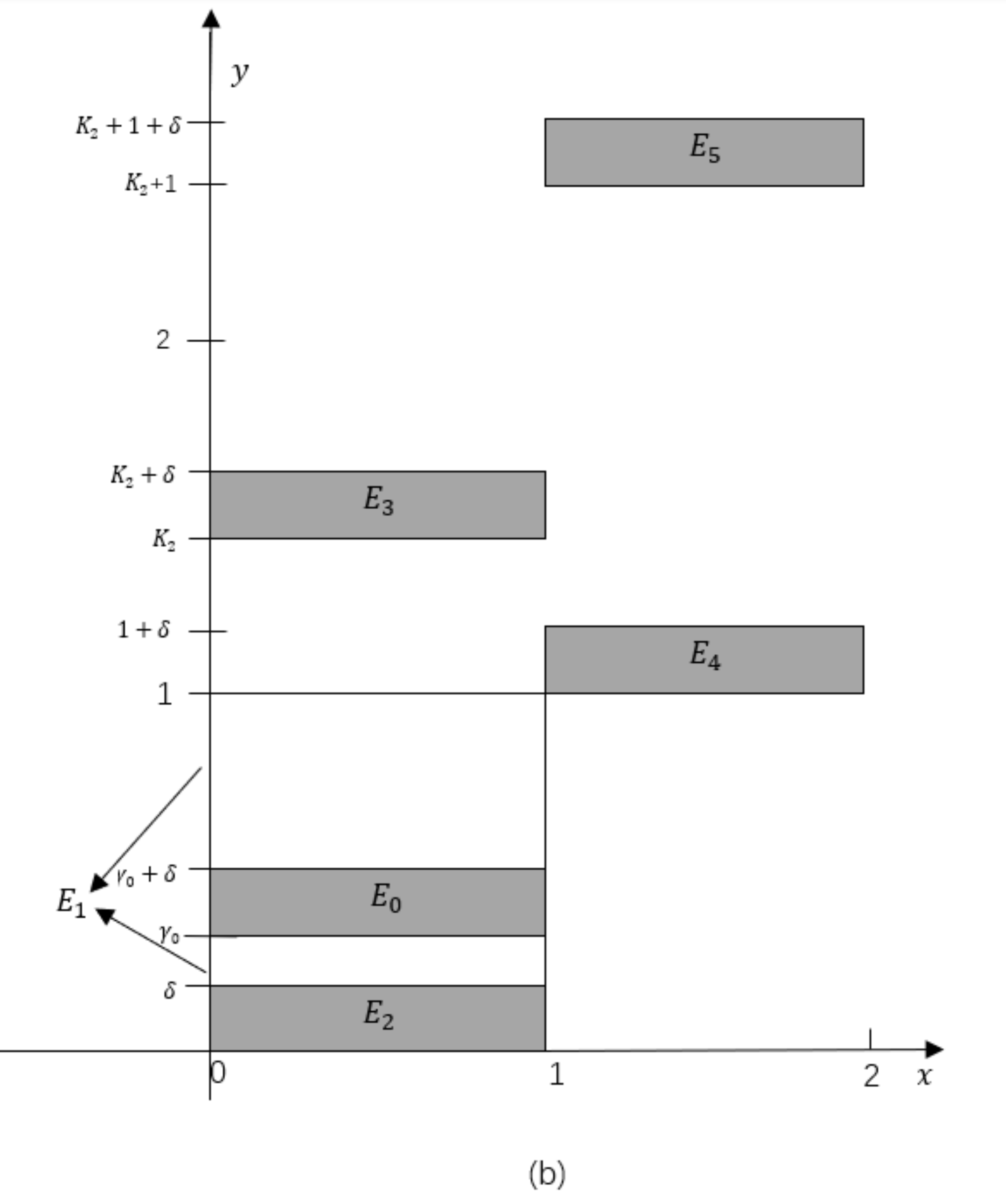}%\hspace{0.2in}
\end{center}
\caption{The sets $E_1$ to $ E_5$ corresponding to the case $q=1$: (a) is for the pair $(\l_2,\k_2) $  and (b) is for the pair $(\l_1,\k_2)$.}
\label{1}
\end{figure}

On the other hand, if we define
$E_1$, $E_2$, $E_3$ exactly the same as in the above, and let $E_4=E_2+\l_1+\l_2$ and
$E_5=E_3+\l_1+\l_2=E_4+\k_2$, then $E_1\cup E_2\cup E_3$ tiles $\R^2$ by $\L$ and $E_1\cup E_2\cup E_4$ packs $\R^2$ by $\K$. By Lemma \ref{Lemma7}(i), we have
$h(\x)\overline{h(\x+\k_2)}=h(\x+\l_1)\overline{h(\x+\l_1+\k_2)}$ for any $\x\in \R^2$. See Figure \ref{1} for an illustration of these sets.

\medskip
\noindent
(ii) $r_0=0$ hence $r=q\ge 2$ is an integer and $\k_2\in \L$. It is easy to verify that the following sets satisfy condition (ii) of Lemma \ref{Lemma6} for the pair $\{\l_2,\k_2\}$: $E_1=\emptyset$, $E_2=[1,0)\times [0,1)$, $E_3=E_2+\k_2$, $E_4=E_2+\l_2$, $E_5=E_3+\l_2$. So we have that $h(\x)\overline{h(\x+\k_2)}=h(\x+\l_2)\overline{h(\x+\l_2+\k_2)}$ by Lemma \ref{Lemma6}.

On the other hand, let $E_1=\emptyset$, $E_2=[1,0)\times [0,1)$, $E_3=E_2+\k_2$, $E_4=E_2+\l_1+\l_2$, $E_5=E_3+\l_1+\l_2$. Then $E_1\cup E_2=E_2$ tiles $\R^2$ by $\L$ and $E_1\cup E_2\cup E_4=E_2\cup E_4$ packs $\R^2$ by $\K$. By Lemma \ref{Lemma7}(ii), we have
$h(\x)\overline{h(\x+\k_2)}=h(\x+\l_1)\overline{h(\x+\l_1+\k_2)}$ for any $\x\in \R^2$. 
\qed

\subsection{Type II Case} That is, $D=\begin{bmatrix}r_1&0\\0&r_2\end{bmatrix}$ where $|r_j|>1$. WLOG assume $r_1>1$ and $r_2>1$.
Let $r_1=q_1+r^\p_0$, $r_2=q_2+r^\pp_0$ with $q_1$, $q_2\in \Z$, $0\le r^\p_0<1$ and $0\le r^\pp_0<1$. There are four cases to consider: (a) $r^\p_0=r^\pp_0=0$; (b) $r^\p_0r^\pp_0>0$; (c) $r^\p_0=0$, $r^\pp_0>0$ and (d) $r^\p_0>0$, $r^\pp_0=0$.

\medskip\noindent
(a) Since $\k_1=q_1\l_1,\k_2=q_2\l_2\in \L$ with $q_1, q_2\ge 2$, the sets $E_1=\emptyset$, $E_2=[1,0)\times [0,1)$, $E_3=E_2+\k_j$, $E_4=E_2+\l_i$, $E_5=E_3+\l_i$ satisfy condition (ii) of Lemma \ref{Lemma6} for any pair of $i$, $j$ ($1\le i,j\le 2$). Thus the result follows from Lemmas \ref{Lemma5} and \ref{Lemma6}.

\medskip\noindent
(b) Let $\delta>0$ be small enough so that $\delta<\min\{r_0^\p,r_0^\pp,1-r_0^\p,1-r_0^\pp\}$, and let $E_0^\p= [r_0^\p,r^\p_0+\delta)\times [0,1)$ and $E_0^\pp=[0,1)\times [r_0^\pp,r_0^\pp+\delta)$.

\smallskip
\noindent
For the pair $\{\l_1,\k_1\}$: $E_2= [0,\delta)\times [0,1)$, $E_1=[1,0)\times [0,1)-(E^\p_0\cup E_2)$, $E_3=E_2+\k_1$, $E_4=
E_2+\l_1$, $E_5=E_3+\l_1$. Notice that $E^\p_0\cup E_1\cup E_2=[1,0)\times [0,1)$, $E_3-q_1\l_1=E^\p_0$ hence $E_1\cup E_2\cup E_3$ tiles $\R^2$ by $\L$ and
$E_1\cup E_2\cup E_4$ packs $\R^2$ by $\K$. So we have $h(\x)\overline{h(\x\pm\k_1)}=h(\x\pm\l_1)\overline{h(\x\pm\l_1\pm\k_1)}$ by Lemma \ref{Lemma6}.

\smallskip
\noindent
 For the pair $\{\l_2,\k_1\}$: $E_1$, $E_2$ and $E_3$ are the same as defined in the above and $E_4=E_2+\l_2+\l_1$, $E_5=E_3+\l_2+\l_1=E_4+\k_1.$ $E_1\cup E_2\cup E_3$ tiles $\R^2$ by $\L$ and $(E_1\cup E_2\cup E_4)$ packs $\R^2$ by $\K.$ By Lemma \ref{Lemma7}(i), $h(\x)\overline{h(\x+\k_1)}=h(\x+\l_2)\overline{h(\x+\l_2+\k_1)}$ for any $\x\in \R^2$.

\smallskip
\noindent
For the pair $\{\l_2,\k_2\}$:
$E_1$ through $E_5$ are defined similarly to their counterparts in the Type I Case with $r$ and $r_0$ replaced by $r_2$ and $r_0^\pp$.
That is, let $\delta>0$ be small enough so that $\delta<min\{r_0^\pp, 1-r_0^\pp\}$ and
let $E_0=[0,1)\times [r_0^\pp,r_0^\pp+\delta),\,E_2=[0,1)\times [0,\delta),$
and $E_1=[0,1)\times [0,1)-(E_0\cup E_2).$ Then $E_0,E_1,E_2$ are disjoint and $E_0\cup E_1\cup E_2=[0,1)\times [0,1)$ tiles $\R^2$ by
$\L$ and packs $\R^2$ by $\K.$ Let $E_3=E_2+\k_2,\,E_4=E_2+\l_2,$ and $E_5=E_3+\l_2.$ Since $E_3-q_2\l_2=E_0$,
$E_1,E_2,E_3$ are disjoint and $E_1\cup E_2\cup E_3$ tiles $\R^2$ by
$\L$ and $E_1\cup E_2\cup E_4$ packs $\R^2$ by $\K.$ So we have that $h(\x)\overline{h(\x+\k_2)}=h(\x\pm\l_2)\overline{h(\x\pm\l_2+\k_2)}$ by Lemma
\ref{Lemma6}(i) and Lemma \ref{Lemma5}.

\smallskip
\noindent
For the pair $\{\l_1,\k_2\}$: let $E_1$, $E_2$, $E_3$ be exactly the same as the above, and $E_4=E_2+\l_1+\l_2$, $E_5=E_3+\l_1+\l_2=E_4+\k_2.$
Then $(E_1\cup E_2\cup E_3)$ tiles $\R^2$ by $\L$ and $(E_1\cup E_2\cup E_4)$ packs $\R^2$ by $\K.$ The result now follows from Lemma \ref{Lemma7}(i).

\medskip\noindent
(c) We have $\k_1\in \L$. Let $\delta>0$ be small enough such that $\delta<\min\{r_0^\pp,1-r_0^\pp\}$.

\smallskip
 For the pair $\{\l_1,\k_1\}$: Let $E_1=\emptyset,\,E_2=[0,1)\times [0,1),\,E_3=E_2+\k_1,\,E_4=E_2+\l_1,\,E_5=E_3+\l_1.$
 So $E_1\cup E_2$ tiles $\R^2$ by $\L,$ and $E_1\cup E_2\cup E_4$ packs $\R^2$ by $\K.$
 Thus $h(\x)\overline{h(\x+\k_1)}=h(\x+\l_1)\overline{h(\x+\l_1+\k_1)}$ by (ii) of Lemma \ref{Lemma6}.

\smallskip
 For the pair  $\{\l_2,\k_1\}$: Let $E_1=\emptyset,\,E_2=[0,1)\times [0,1),\,E_3=E_2+\k_1,\,E_4=E_2+\l_2+\l_1,\,E_5=E_3+\l_2+\l_1=E_4+\k_1.$ So $E_1\cup E_2$ tiles $\R^2$ by $\L,$ and $E_1\cup E_2\cup E_4$ packs $\R^2$ by $\K.$ Thus $h(\x)\overline{h(\x+\k_1)}=h(\x+\l_2)\overline{h(\x+\l_2+\k_1)}$ by (ii) of Lemma \ref{Lemma7} and the case for the pair $\{\l_1,\k_1\}$ above. 
 
\smallskip
 For the pair $\{\l_1,\k_2\}$: Let $\delta>0$ be small enough so that
 $\delta<min\{r_0^\pp,\,1-r_0^\pp\}$, and let
 $E_0^\pp=[0,1)\times[r_0^\pp,r_0^\pp+\delta),\,E_2=[0,1)\times[0,\delta),\,
 E_1=[0,1)\times[0,1)-(E_0^\pp\cup E_2),\,E_3=E_2+\k_2,\,E_4=E_2+\l_1,\,E_5=E_3+\l_1.$
Then $E_1\cup E_0^\pp\cup E_2=[0,1)\times[0,1)$ tiles $\R^2$ by $\L$ and packs $\R^2$ by $\K$.
$E_3-q_2\l_2=E_0^\pp,$ so $E_1\cup E_2\cup E_3$ tiles $\R^2$ by $\L,$ and $(E_1\cup E_2\cup E_4)$ packs $\R^2$ by $\K.$
 Therefore  the sets $E_1$ to $E_5$ satisfy condition (i) of Lemma \ref{Lemma6}, and it follows that
 $h(\x)\overline{h(\x+\k_2)}=h(\x+\l_1)\overline{h(\x+\l_1+\k_2)}$.

\smallskip
Finally for the pair  $\{\l_2,\k_2\}$,  let
 $E_0^\pp=[0,1)\times[r_0^\pp,r_0^\pp+\delta),\,E_2=[0,1)\times[0,\delta),\,
 E_1=[0,1)\times[0,1)-(E_0^\pp\cup E_2),\,E_3=E_2+\k_2,\,E_4=E_2+\l_2,\,E_5=E_3+\l_2.$
It is easy to see that $E_1\cup E_2\cup E_3$ tiles $\R^2$ by $\L,$ and $(E_1\cup E_2\cup E_4)$ packs $\R^2$ by $\K.$ The result follows from Lemma \ref{Lemma6}(i). See Figure \ref{4}.

\medskip
\begin{figure}[ht!]
\begin{center}
\includegraphics[height=2.3in,width=2.4in]{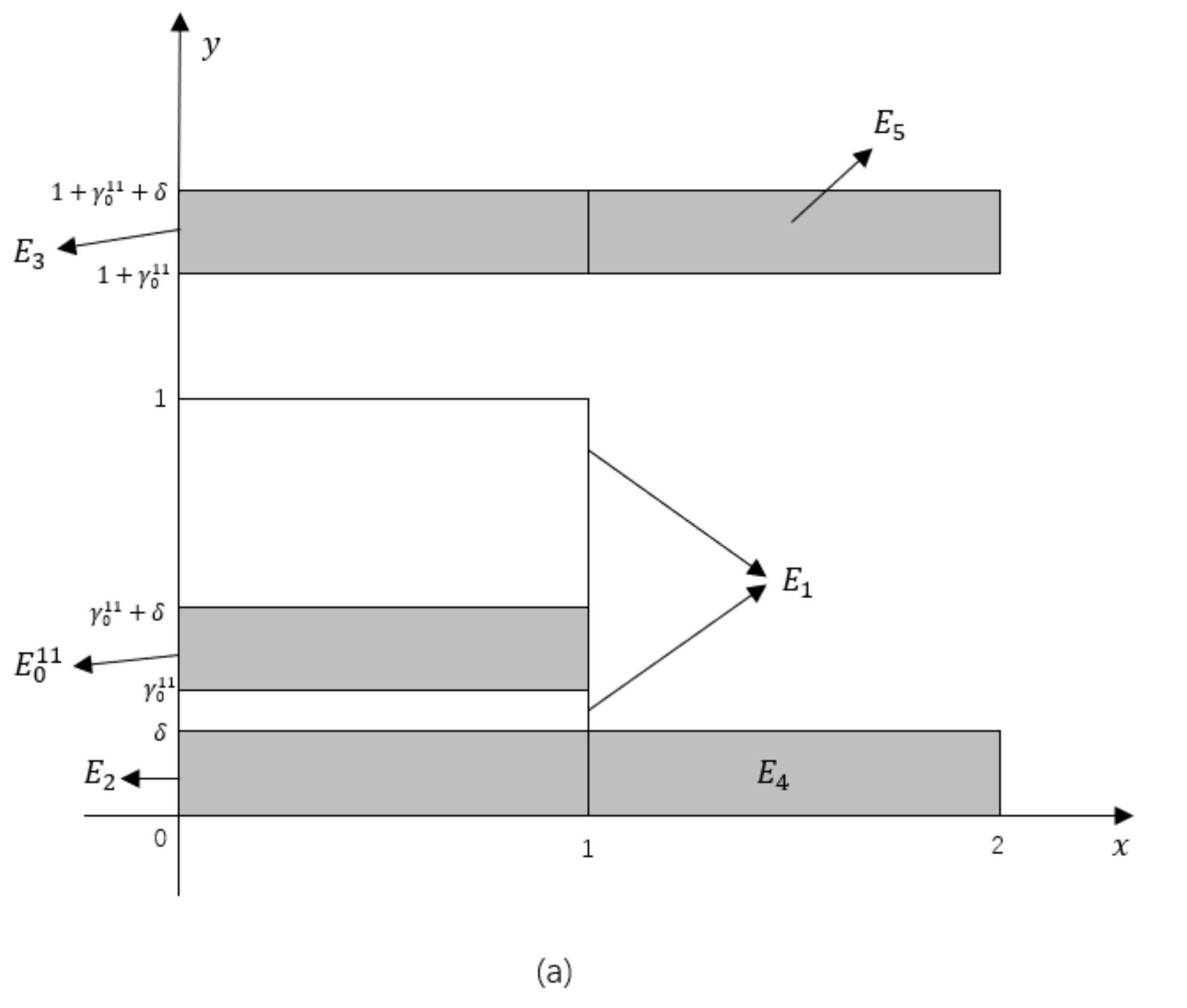}\hspace{0.1in}
\includegraphics[height=2.3in,width=1.6in]{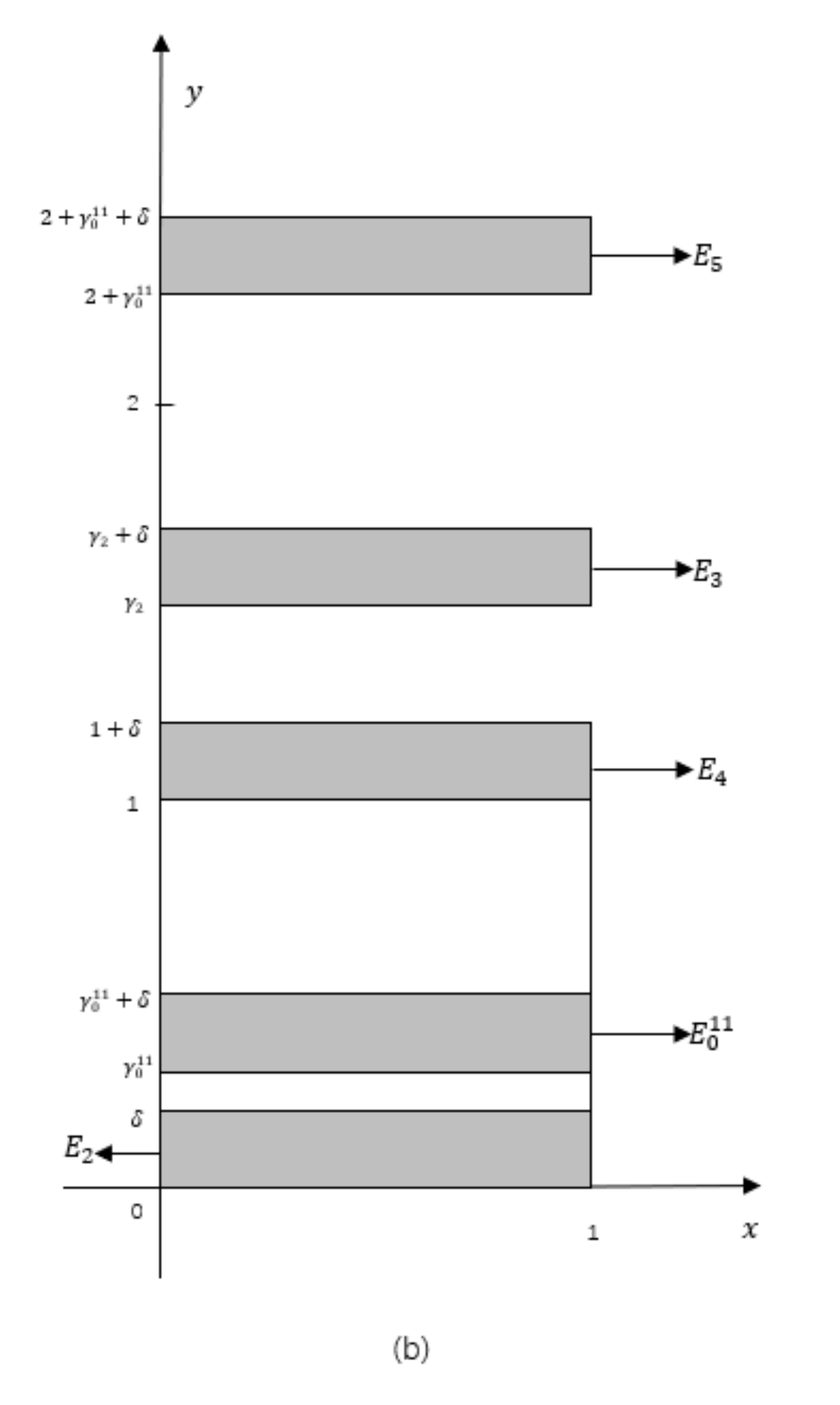}\hspace{0.2in}
\end{center}
\caption{The sets $E_1$ to $ E_5$ corresponding to $q_2=1$: (a) is for the pair $(\l_1,\k_2) $ and (b) is for the pair $(\l_2,\k_2)$.}
\label{4}
\end{figure}

\medskip\noindent
(d) We have $\k_2\in \L$. Let $\delta>0$ be small enough such that $\delta<\min\{r_0^\p,1-r_0^\p\}$. For the pairs $\{\l_1,\k_1\}$ and $\{\l_2,\k_1\}$, the sets $E_1$ to $E_5$ defined the same way as as their counterparts in the case of (b) satisfy condition (i) of Lemma \ref{Lemma6}. See Figure \ref{5}.

\medskip
\begin{figure}[ht!]
\begin{center}
\includegraphics[scale=0.3]{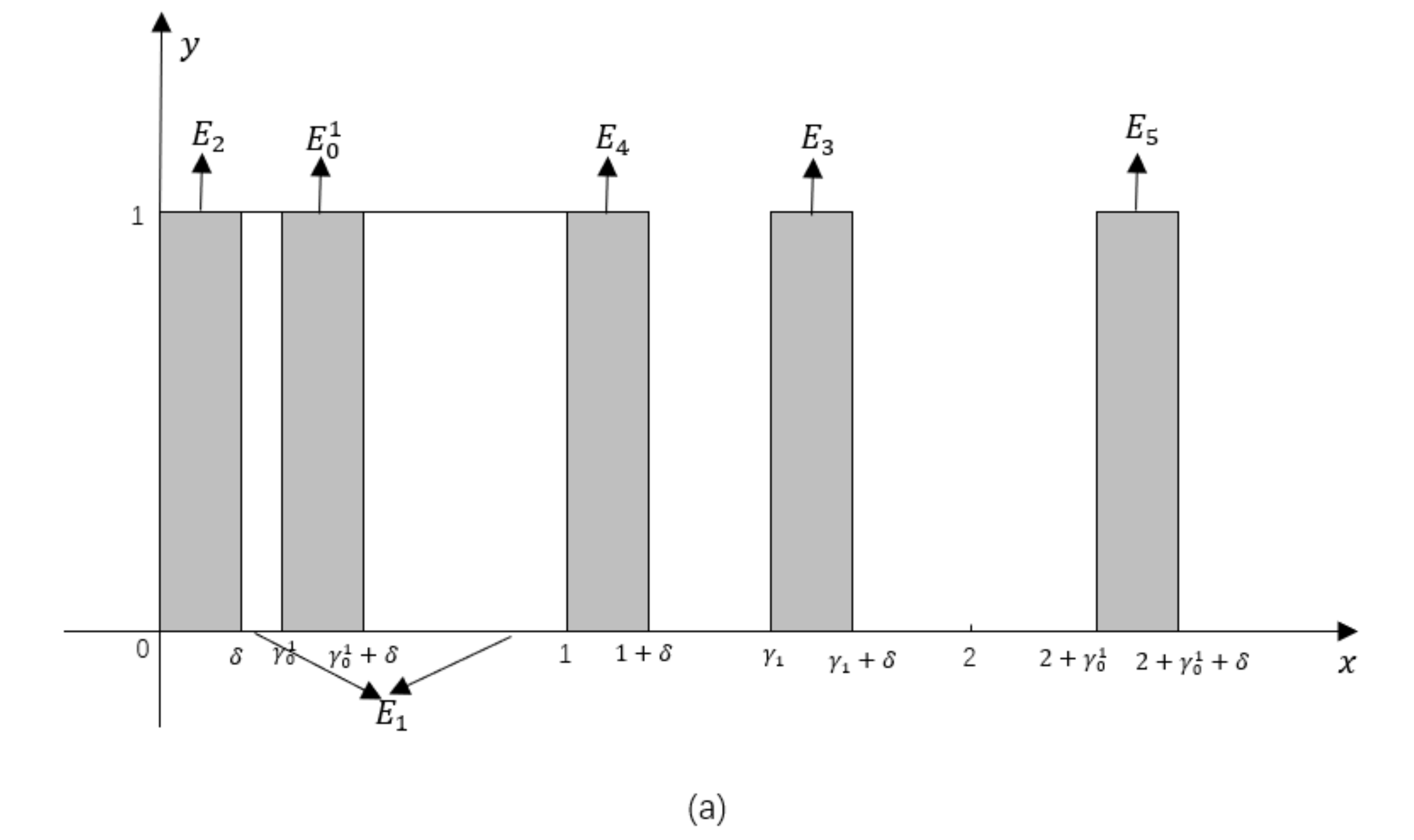}\hspace{0.1in}
\includegraphics[scale=0.3]{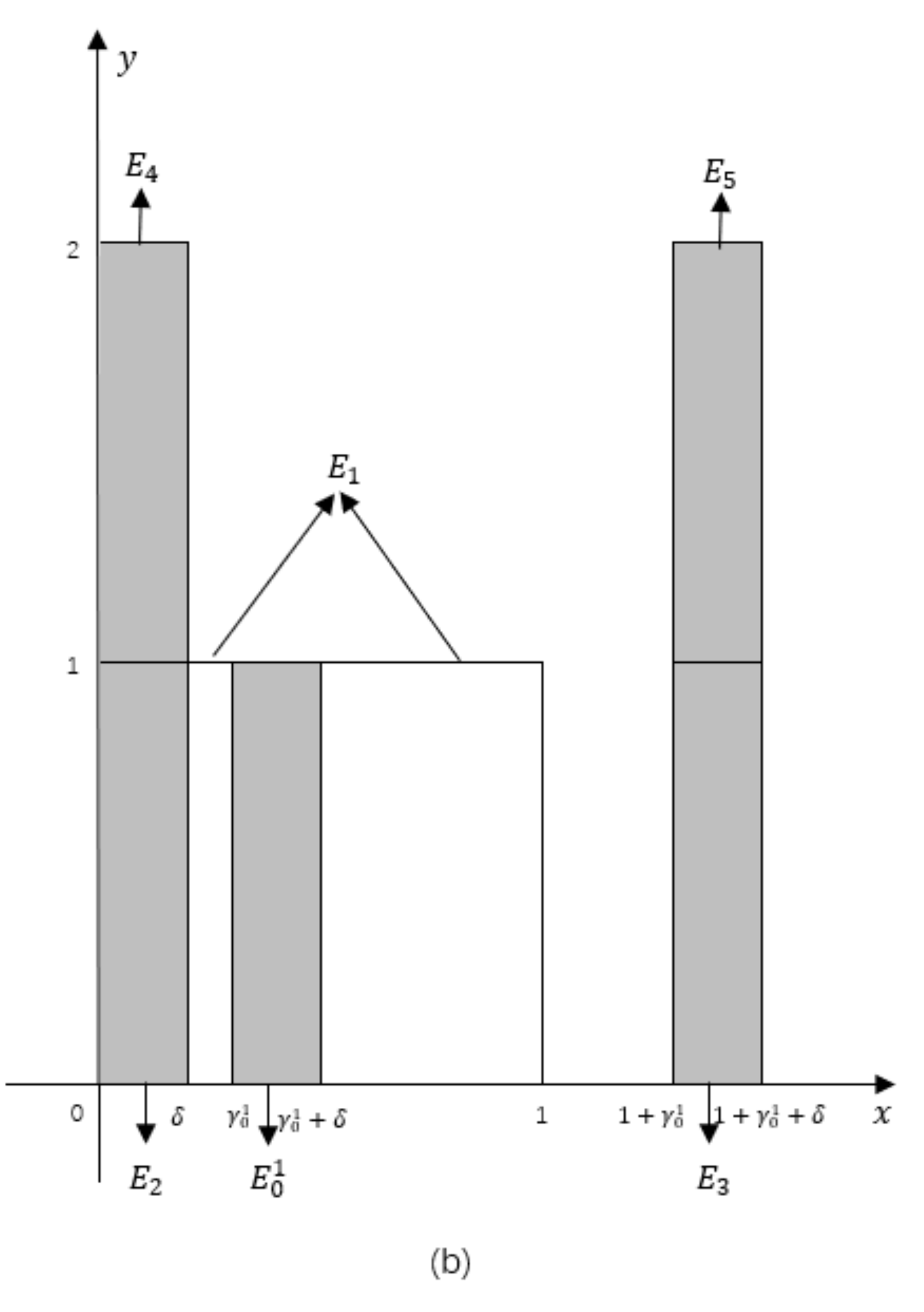}\hspace{0.2in}
\end{center}
\caption{The sets $E_1$ to $ E_5$ corresponding to $q_1=1$: (a) is for the pair $(\l_1,\k_1) $ and (b) is for the pair $(\l_2,\k_1)$.}
\label{5}
\end{figure}

\smallskip
For the pair $\{\l_2,\k_2\}$, the sets $E_1=\emptyset,\, E_2=[0,1)\times[0,1),\,E_3=E_2+\k_2,\, E_4=E_2+\l_2,\,E_5=E_3+\l_2=E_4+\k_2$ satisfy condition (ii) of Lemma \ref{Lemma6}, hence $h(\x)\overline{h(\x+\k_2)}=h(\x+\l_2)\overline{h(\x+\l_2+\k_2)}.$

 \smallskip
For the last pair $\{\l_1,\k_2\}$, the sets $E_1=\emptyset,\, E_2=[0,1)\times[0,1),\,E_3=E_2+\k_2,\, E_4=E_2+\l_1+\l_2,\,E_5=E_3+\l_1+\l_2=E_4+\k_2$ satisfy condition (ii) of Lemma \ref{Lemma7} as one can easily check. Combining it with the case for the pair $\{\l_2,\k_2\}$ above leads to  $h(\x)\overline{h(\x+\k_2)}=h(\x+\l_1)\overline{h(\x+\l_1+\k_2)}$ for all $\x\in\R^2$.

 \medskip
\subsection{Type III Case} That is, $D=\begin{bmatrix}1&0\\0&1\end{bmatrix}$. Let $U_1$ and $U_2$ be the bilateral-shift unitary operator on $\ell^2(\Z\times \Z)$ acting on the first and the second coordinator respectively. That is, if $\{e_{(n,m)}\}$ is the standard orthonormal basis for
$\ell^2(\Z\times \Z)$ (that is, $e_{(n,m)}=\{a_{i,j}\}\in \ell^2(\Z\times \Z)$ with $a_{n,m}=1$ and $a_{i,j}=0$ if $\{i,j\}\not=\{n,m\}$), then $U_1e_{(n,m)}=e_{(n+1,m)}$ and $U_2e_{(n,m)}=e_{(n,m+1)}$. For any $\eta=\{\eta_{i,j}\}\in \ell^2(\Z\times \Z)$ (we shall write $\eta_{i,j}=\eta_\l$ for short with $\l=i\l_1+j\l_2$) such that $\{U_1^nU_2^m\eta: n,m\in \Z\}$ is an orthonormal basis of $\ell^2(\Z\times \Z)$, define $g$ so that $g(\x-\l)=\sqrt{d_0}\eta_{\l}$ for any $\x\in [0,1)\times [0,1)$. We have
$$
\sum_{\l\in\L}g(\x-\l)\overline{g(\x-\l-\k)}=d\sum_{\l\in\L}\eta_{\l}\overline{\eta_{\l+\k}}=d_0\delta_{\0,\k}
$$
since $\{U_1^nU_2^m\eta: n,m\in \Z\}$ is an orthonormal basis of $\ell^2(\Z\times \Z)$. Thus (\ref{cond12}) and (\ref{cond22}) hold for any $\x$ since $[0,1)\times [0,1)$ tiles $\R^2$ by $\L=\K$.
This proves that $g$ is a normalized tight Gabor frame function. Let $h$ be a functional multiplier, we have
\begin{eqnarray*}
\sum_{\l\in\L}(hg)(\x-\l)\overline{(hg)(\x-\l-\k)}=d_0\sum_{\l\in\L}h(\x-\l)\overline{h(\x-\l-\k)}\eta_{\l}\overline{\eta_{\l+\k}}=d_0\delta_{\0,\k}
\end{eqnarray*}
for any $\k\in \L=\K$. This is equivalent to the condition that $\{U_1^nU_2^mM_{\tilde{h}}\eta: n,m\in \Z\}$ is an orthonormal basis of $\ell^2(\Z\times \Z)$, where $\tilde{h}=\{h(\x-\l): \l\in \L\}\in \ell^2(\Z\times \Z)$ for each fixed $\x$. Thus $M_{\tilde{h}}$ is a wandering vector multiplier for the group
$\{U_1^nU_2^m: (n,m)\in \Z\times \Z\}$. Therefore by \cite{HL2001} Theorem 3.1 (i), there exist $\lambda\in \C$ and $\s\in [0,1)\times [0,1)$ that depend only on $\x$ such that $|\lambda|=1$ and $h(\x-\l)=\lambda e^{2\pi i\langle \l,\s\rangle}$ for all $\l\in \L$. From this it is easy to verify that $h(\x)\overline{h(\x-\k)}=h(\x-\l)\overline{h(\x-\l-\k)}$ for any $\l, \k\in \L=\K$.

\medskip
\subsection{Type IV Case} That is, $D=\begin{bmatrix}r&1\\0&r\end{bmatrix}$ with $|r|>1$. WLOG let us assume that $r>1$. Let $r=q+r_0$ with $q\in \Z$ and $0\le r_0<1$. There are two cases: (a) $r_0=0$ and (b) $r_0>0$.

\medskip
(a) $r_0=0$ hence $r=q\ge 2$. $E_1=\emptyset$, $E_2= [0,1)\times [0,1)$, then $E_2$ tiles $\R^2$ by $\L$ and packs $\R^2$ by $\K$. Thus for any pair $\{\l_i,\k_j\}$ ($1\le i,j\le 2$) we can simply define $E_3=E_2+\k_j$, $E_4=E_2+\l_i$ and $E_5=E_3+\l_i$. $E_1$ through $E_5$ satisfy condition (ii) of Lemma \ref{Lemma6}.

\medskip
(b) $r_0>0$. Let $\delta>0$ be small enough such that $\delta<\min\{r_0,1-r_0\}$. Notice that $[0,1)\times [0,1)$ tiles $\R^2$ by $\L$ and packs $\R^2$ by $\K$.
For $\{\l_1,\k_1\}$ and $\{\l_2,\k_1\}$, let $E_0=[r_0,r_0+\delta)\times [0,\delta)$, $E_2=[0,\delta)\times [0,\delta)$, $E_1=[0,1)\times [0,1)-(E_0\cup E_2)$,then $E_0\cup E_1\cup E_2=[0,1)\times [0,1)$, $\mu(E_2)>0$. Let $E_3=E_2+\k_1=[r,r+\delta)\times [0,\delta)$, $E_4=E_2+\l_i$ and $E_5=E_3+\l_i$, $i=1,2$. Then as before we have $E_3=E_0+q\l_1$ hence $E_1\cup E_2\cup E_3$ tiles $\R^2$ by $\L$ and $E_1\cup E_2\cup E_4$ packs $\R^2$ by $\K$. Thus $E_1$ through $E_5$ satisfy condition (i) of Lemma \ref{Lemma6}.
For $\{\l_1,\k_2\}$ and $\{\l_2,\k_2\}$, we let $E_2=[0,\delta)\times [0,\delta)$, $E_1=[0,1)\times [0,1)-E_2$, $E_3=E_2+\k_2$, $E_4=E_2+\l_i$, $E_5=E_3+\l_i$, $i=1,2$. $E_1$ through $E_5$ satisfy condition (ii) of Lemma \ref{Lemma6}.

\medskip
\subsection{Type V Case} That is, $D=\begin{bmatrix}1&1\\0&1\end{bmatrix}$ or $D=\begin{bmatrix}-1&1\\0&-1\end{bmatrix}$. WLOG assume that $D=\begin{bmatrix}1&1\\0&1\end{bmatrix}$. We have $\k_1=\l_1$ and $\k_2=\l_1+\l_2$ hence $\k_2^\p=\k_2-\k_1=\l_2$ and the result of Type III applies to $\k_1^\p=\k_1$ and $\k_2^\p=\k_2-\k_1$. The result follows by Remark \ref{Remark7}.

\section{The proof of the main result: Part 2}\label{Sec_Proof2}

We now prove Theorem \ref{Theorem1} for the case when one of the eigen values of $D$ has absolute value less than one. Since $|\det(D)|\ge 1$, this implies that the other eigen value has absolute value bigger than one, therefore $D$ must be a diagonal matrix. Without loss of generality we will assume that $D=\begin{bmatrix}r_1& 0\\
0 & r_2\end{bmatrix}$ with $0<r_1<1<r_2$. We will divide $D$ into the following types and consider these types one by one as we did in the last section.

\medskip
\noindent
Type VI: $r_1$ and $r_2$ are both irrational;\\
Type VII: $r_1$ is rational and $r_2$ is irrational; \\
Type VIII: $r_2$ is rational and $r_1$ is irrational; \\
Type IX: $r_1$ and $r_2$ are both rational.

\textbf{Type VI Case}: both $r_1$ and $r_2$ are irrational. In this case, $\M=\{\l+\k:\ \l\in \L, \k\in\K\}$ is a dense subset of $\R^2$ and Lemma \ref{newlemma1} applies. In each of the following cases, it is understood that the choice of the number $\delta>0$ is such that the sets $E_1$ to $E_5$ so constructed are all disjoint. This can be done when $\delta$ is small enough.

\medskip
For the pair $(\l_2,\k_1)$, subdivide $S$ and $R$ into disjoint unions of rectangles such that $E_2=[0,\delta)\times [0,\delta)$ is a common rectangle in both partitions, $E_2+\k_1\subset S$ is a rectangle in the partition of $S$ and $E_2+\l_2\subset R$ is the rectangle in the partition of $R$ corresponding to  $E_2+\k_1\subset S$ in Lemma \ref{newlemma1}. Thus we can choose $\O$ so that $E_5=E_2+\k_1+\l_2\subset \O$ since it is $\L$ equivalent to $E_2+\k_1$ and $\K$ equivalent to $E_2+\l_2$. It follows that the sets $E_1=\O\setminus\big(E_2\cup E_5)$, $E_3=E_2+\k_1$, $E_4=E_2+\l_2$ and $E_5=E_2+\k_1+\l_2$ satisfy condition (i) of Lemma \ref{Lemma6}.

\medskip
For the pair $(\l_1,\k_1)$, let $q^\p>0$ be the largest integer such that $q^\p r_1<1$. It follows that $q^\p r_1+r_1=(q^\p+1)r_1>1$ hence $0<1-q^\p r_1<r_1$. Choose $\delta>0$ small enough such that $\delta< 1-q^\p r_1$ and $\delta+1-q^\p r_1<r_1$.
Subdivide $S$ and $R$ into disjoint unions of rectangles such that $E_2=[0,\delta)\times [0,\delta)$ is a common rectangle in both partitions, $E_2+\k_1\subset S$ is in the partition of $S$ and $E_2+\l_1-q^\p \k_1\subset R$ is in the partition of $R$. By Lemma \ref{newlemma1}, we can construct $\O$ so that $E_2=[0,\delta)\times [0,\delta)\subset \O$ and $\O$ contains a subset $E_0$ that is $\L$ equivalent to $E_2+\k_1\subset S$ and $\K$ equivalent to $E_2+\l_1-q^\p \k_1\subset R$ (which is $\K$ equivalent to $E_2+\l_1$) Then the sets $E_2=[0,\delta)\times [0,\delta)$, $E_1=\O\setminus\big(E_2\cup E_0)$, $E_3=E_2+\k_1$, $E_4=E_2+\l_1$ and $E_5=E_2+\k_1+\l_1$ satisfy condition (i) of Lemma \ref{Lemma6}.

\medskip
For the pair $(\l_2,\k_2)$, subdivide $S$ and $R$ into disjoint unions of rectangles such that x$E_2=[0,\delta)\times [0,\delta)$ is a common rectangle in both partitions, $E_2+r_0\l_2\subset S$ is in the partition of $S$ and $E_2+\l_2\subset R$ is in the partition of $R$ (where $r_0=r_2-\lfloor r_2\rfloor$). By Lemma \ref{newlemma1}, we can construct $\O$ so that $E_2=[0,\delta)\times [0,\delta)\subset \O$ and $\O$ contains a subset $E_0$ that is $\L$ equivalent to $E_2+r_0\l_2$ and $\K$ equivalent to $E_2+\l_2$. Then the sets $E_1=\O\setminus(E_2\cup E_0)$, $E_2$, $E_3=E_2+\k_2$, $E_4=E_2+\l_2$ and $E_5=E_2+\k_2+\l_2$ satisfy condition (i) of Lemma \ref{Lemma6} since $E_3$ is $\L$ equivalent to $E_2+r_0\l_2$, hence to $E_0$.

\medskip
For the pair $(\l_1,\k_2)$, subdivide $S$ and $R^\p=([0,r_1)\times [0,\delta))\cup ([1,1+r_1)\times [\delta,r_2))$ into disjoint unions of rectangles such that $E_2=[0,\delta)\times [0,\delta)$ is a common rectangle in both partitions, $E_2+r_0\l_2\subset S$ is in the partition of $S$ and
 $E_2+\l_1+\l_2\subset R^\p$ is in the partition of $R^\p$. By Lemma \ref{newlemma1}, we can construct $\O$ so that $E_2=[0,\delta)\times [0,\delta)\subset \O$ and $\O$ contains a subset $E_0$ that is $\L$ equivalent to $E_2+r_0\l_2$ and $\K$ equivalent to $E_2+\l_1+\l_2$. Then the sets $E_1=\O\setminus(E_2\cup E_0)$, $E_2$, $E_3=E_2+\k_2$, $E_4=E_2+\l_1+\l_2$ and $E_5=E_2+\k_2+\l_1+\l_2$, together with the result for the pair $(\l_2,\k_2)$ above, satisfy condition (i) of Lemma \ref{Lemma7}.

\medskip
\textbf{Type VII Case}:  $r_1$ is rational and $r_2$ is irrational. 

The proof in this case is similar to that of the Type VI case, with slight modifications using Lemma  \ref{newlemma3} instead of Lemma \ref{newlemma1}. The only exception here is the case when $m_1=1$ for the pairs $(\l_1,\k_1)$ and $(\l_1,\k_2)$. Thus we shall only provide the proof for the pairs $(\l_1,\k_1)$ and $(\l_1,\k_2)$ when $r_1=m_1/n_1=1/n_1$ for some integer $n_1$, and leave the other cases for our reader to verify. 

\medskip
For the pair $(\l_1,\k_2)$,  choose $\delta>0$ small enough so that $\delta<\min\{1-r_0,r_2-1\}$ where $r_0=r_2-\lfloor r_2\rfloor$. Subdivide $S$ and $R^\p=([0,1/n_1)\times [0,\delta))\cup ([1,1+1/n_1)\times [\delta,r_2))$ into disjoint unions of rectangles such that $E_2=[0,1/n_1)\times [0,\delta)$ is a common rectangle in both partitions, $E_2+r_0\l_2\subset S$ is in the partition of $S$ and
 $E_2+\l_1+\l_2\subset R^\p$ is in the partition of $R^\p$. By Lemma \ref{newlemma3}, we can choose $\O$ so that $E_2=[0,1/n_1)\times [0,\delta)\subset \O$ and $\O$ contains a subset $E_0$ that is $\L$ equivalent to $E_2+r_0\l_2\subset S$ and $\K$ equivalent to $E_2+\l_1+\l_2$. Then the sets $E_1=\O\setminus(E_2\cup E_0)$, $E_2$, $E_3=E_2+\k_2$, $E_4=E_2+\l_1+\l_2$ and $E_5=E_2+\k_2+\l_1+\l_2$, together with the result for the pair $(\l_2,\k_2)$, satisfy condition (i) of Lemma \ref{Lemma7}.

\medskip
For the pair $(\l_1,\k_1)$, choose $\delta>0$ small enough so that $\delta<r_2-1$. Subdivide $S$ and $R^\pp=([0,r_1)\times [0,\delta))\cup ([r_1,2r_1)\times [\delta,r_2))$ into disjoint unions of rectangles such that $E_2=[0,1/n_1)\times [0,\delta)$ is a common rectangle in both partitions, $E_2+\k_1\subset S$ is in the partition of $S$ and $E_0=E_2+\k_1+\l_2\subset R^\pp$ is in the partition of $R^\pp$. By Lemma \ref{newlemma3}, we can choose  $\O$ so that $E_2\subset \O$ and $E_0\subset \O$ (since $E_0$ is $\L$ equivalent to $E_2+\k_1$ and $\K$ equivalent to itself). Notice that $E_4=E_2+\l_1+\l_2=E_0+(n_1-1)\k_1$ hence it is $\K$ equivalent to $E_0$. Thus the sets $E_1=\O\setminus\big(E_2\cup E_0)$, $E_3=E_2+\k_1$, $E_4=E_2+\l_1+\l_2$ and $E_5=E_2+\k_1+\l_1+\l_2$, together with the result for the pair $(\l_2,\k_1)$, satisfy condition (i) of Lemma \ref{Lemma7}.

\medskip
\textbf{Type VIII Case}: $r_1$ is irrational and $r_2$ is rational. 

\medskip
As in the Type VII case, the proof in this case is also similar to that of the Type VI case. The only exception here is for the pairs $(\l_2,\k_2)$ and $(\l_1,\k_2)$ when $r_2=m$ is an integer to which we will provide the proof while leaving the other cases to our reader as exercises. 

\medskip
Notice that if $r_2=m$ is an integer, then it is necessary that $m>1$. Furthermore, $mr_1>1$ hence $r_1>1/m$. Let $q^\p$ be the largest integer such that $q^\p r_1<1$. For the pair $(\l_2,\k_2)$, partition $S$ as $S=\cup_{0\le j\le m-1}C_j$ where $C_j=[j/m,(j+1)/m)\times [0,1)$, and divide $R=[0,r_1)\times [0,m)$ such that the partition contains the rectangles $C_j^\p=[r_1-1/m,r_1)\times [j,j+1)$ for $1\le j\le m-1$, as well as the rectangle $C^\p_0=C_0$ so that $C_0=[0,1/m)\times [0,1)$ is a common rectangle in both partitions. By Lemma \ref{newlemma3}, we can choose $\O$ so that $C_0\subset \O$ and for each $1\le j\le m-1$, $\O$ contains a measurable set $C_j^\pp$ that is $\L$ equivalent to $C_j$ (hence to $C_j+j\l_2$) and $\K$ equivalent to $C_j^\p$. Choose $\delta>0$ small enough so that $\delta<\min\{r_1-\frac{1}{m},\frac{1}{m},(1+q^\p )r_1-1\}$. Let $E_2=[0,\delta)\times [0,1)\subset \O$. Since $\O$ is $\K$ equivalent to $\cup_{0\le j\le m-1}C_j^\p\subset R$ which is disjoint from $E_2+\l_2\subset R$ (and $R$ packs $\R^2$ by $\K$), if we let $E_1=\O\setminus E_2$, $E_4=E_2+\l_2$, then $E_1\cup E_2\cup E_4$ packs $\R^2$ by $\K$. Thus $E_1$, $E_2$, $E_4$ so defined, together with $E_3=E_2+\k_2$ and $E_5=E_2+\l_2+\k_2$ satisfy condition (ii) of Lemma \ref{Lemma6}. 

\medskip
For the pair $(\l_1,\k_2)$, we need to consider the following two cases separately: either $1/m\le 1-q^\p r_1$ or $1/m> 1-q^\p r_1$. 

\medskip
In the first case where $1/m\le 1-q^\p r_1<r_1$, choose $\O$ so constructed as for the pair $(\l_2,\k_2)$ above, then the sets $E_2=[0,\delta)\times [0,1)$, $E_1=\O\setminus E_2$, $E_3=E_2+\k_2$, $E_4=E_2+\l_1$ and $E_5=E_2+\l_1+\k_2$ satisfy condition (ii) of Lemma \ref{Lemma6}: $E_4=E_2+\l_1$ is $\K$ equivalent to $E_2+(1-q^\p r_1)\l_1=[1-q^\p r_1,1-q^\p r_1+\delta)\times [0,1)\subset R$, but $\O$ is $\K$ equivalent to the subset $\cup_{0\le j\le m-1}C_j^\p$ of $R$ that is disjoint from it.

\medskip
For the second case where $1/m> 1-q^\p r_1$, choose $\delta>0$ small enough so that $\delta<\min\{1-q^\p r_1, \frac{1}{m}+q^\p r_1-1,r_1-\frac{1}{m}\}$. We will then divide $S$ so that the partition contains the rectangles $C_j=[j/m,(j+1)/m)\times [0,1)$ ($1\le j\le m-1$), as well as the rectangles $C_{01}=[0,1-q^\p r_1)\times [0,1)$ and $C_{02}=
 [1-q^\p r_1, 1-q^\p r_1+\delta)\times [0,1))$ and $C_{03}=
 [1-q^\p r_1+\delta, \delta+1/m)\times [0,1))$. On the other hand, we will divide $R$ so that the partition contains the rectangles $C_j^\p=[r_1-1/m,r_1)\times [j,j+1)$ for $1\le j\le m-1$, as well as the rectangles $C^\p_{01}=C_{01}$ and $C^\p_{02}=
 [0, \delta)\times [1,2))$ and $C^\p_{03}=C_{03}$. By Lemma  \ref{newlemma3}, we can choose $\O$ so that $C_{01}= C_{01}^\p\subset \O$, $C_{03}= C_{03}^\p\subset \O$ and that $\O=\big(\cup_{0\le j\le m-1}C^\pp_j\big)\cup C_{02}^\pp$ where $C_j^\pp$ is $\L$ equivalent to $C_j$ and $\K$ equivalent to $C_j^\p$ for each $j$, and $C_{02}^\pp$ is $\L$ equivalent to $C_{02}$ and $\K$ equivalent to $C_{02}^\p$. Let $E_2=[0,\delta)\times [0,1)$, $E_1=\O\setminus E_2$, $E_3=E_2+\k_2$, $E_4=E_2+\l_1$ and $E_5=E_2+\l_1+\k_2$. Notice that $E_4$ is $\K$ equivalent to $E_4-q^\p \k_1=E_2+(1-q^\p r_1)\l_1=C_{02}$, but by its construction, $\O$ is $\K$ equivalent to a subset of $R$ that is disjoint from $C_{02}$. Thus the sets $E_1$ to $E_5$ so defined  satisfy condition (ii) of Lemma \ref{Lemma6}. 

\medskip
\textbf{Type IX Case}: $r_1=m_1/n_1$ and $r_2=m_2/n_2$ are both rational, with $r_1r_2=(m_1m_2)/(n_1n_2)\ge 1$. In this case, by Lemma \ref{newlemma2},
There are four subcases to consider: IX(a) $m_1\not=1$ and $n_2\not=1$; IX(b) $m_1\not=1$, $n_2=1$; IX(c) $m_1=1$, $n_2\not=1$; IX(d) $m_1=1$ and $n_2=1$.

\medskip
IX(a) The proof is identical to the proof of the Type VI case with $E_2$ replaced by $[0,1/n_1)\times [0,1/n_2)$;

\medskip
IX(b) The proofs for the pairs $(\l_1,\k_1)$ and $(\l_2,\k_1)$ are identical to the proof of IX(a) above.  Partition $S$ as $S=\cup_{1\le j\le n_1}C_j$ with  $C_j=[\frac{j-1}{n_1},\frac{j}{n_1})\times [0,1)$ and $R$ as $R=\cup_{1\le i\le m_1, 1\le j\le m_2}([\frac{i-1}{n_1},\frac{i}{n_1})\times [j-1,j))$. Since $r_1r_2=(m_1m_2)/n_1>1$, we have $m_1m_2\ge n_1+1$.

For the pair $(\l_2,\k_2)$, choose $\O$ by not using $E_2+\l_2$ as any of the rectangles $C_j^\p$ in Lemma \ref{newlemma2} (this is possible since $m_1m_2\ge n_1+1$). Thus $\O$ is $\K$ equivalent to a subset of $R$ that is disjoint from $E_2+\l_2$, and the sets $E_1=\O\setminus E_2$, $E_2$, $E_3=E_2+\k_2$, $E_4=E_2+\l_2$ and $E_5=E_2+\l_2+\k_2$ satisfy condition (ii) of Lemma \ref{Lemma6}.

For the pair  $(\l_1,\k_2)$,  let $q_0^\p$ be the largest integer such that $\frac{q_0^\p m_1}{n_1}<1$. choose $\O$ so that it contains $E_2=[0,1/n_1)\times [0,1)$ and is $\K$ equivalent to a subset of $R$ that is disjoint from $E_2+\l_2+(1-\frac{q_0^\p m_1}{n_1})\l_1\subset R$. Then the sets $E_1=\O\setminus E_2$, $E_2$, $E_3=E_2+\k_2$, $E_4=E_2+\l_1+\l_2$ (which is $\K$ equivalent to  $E_2+\l_2+(1-\frac{q_0^\p m_1}{n_1})\l_1$) and $E_5=E_2+\l_1+\l_2+\k_2$, together with the result above about the pair $(\l_2,\k_2)$, satisfy condition (ii) of Lemma \ref{Lemma7}.

\medskip
IX(c) Partition $S$ as $S=\cup_{1\le i\le n_1, 1\le j\le n_2}([\frac{i-1}{n_1},\frac{i}{n_1})\times [\frac{j-1}{n_2},\frac{j}{n_2}))$ and $R$ as $R=\cup_{1\le j\le m_2}([0,\frac{1}{n_1})\times [\frac{j-1}{n_2},\frac{j}{n_2}))$. Since $r_1r_2=m_2/(n_1n_2)\ge 1$, we must have $m_2\ge n_1n_2+1$ since $m_2$ and $n_2>1$ are coprime.

\medskip
Choose $\O$ so that it contains $E_2=[0,1/n_1)\times [0,1/n_2)$ and is $\K$ equivalent to a subset of $R$ that is disjoint from $E_2+\l_2$ (which is in the partition of $R$). Again this is possible since $m_2\ge n_1n_2+1$ hence there are more rectangles in the partition of $R$ than in the partition of $S$ in the above. Then for the pair $(\l_2,\k_2)$, the sets $E_1=\O\setminus E_2$, $E_2$, $E_3=E_2+\k_2$, $E_4=E_2+\l_2$ and $E_5=E_2+\l_2+\k_2$ satisfy condition (ii) of Lemma \ref{Lemma6}. For the pair  $(\l_1,\k_2)$, the sets $E_1=\O\setminus E_2$, $E_2$, $E_3=E_2+\k_2$, $E_4=E_2+\l_1+\l_2$ (which is $\K$ equivalent to $E_2+\l_2$ since $\l_1=n_1\k_1\in \K$) and $E_5=E_2+\l_1+\l_2+\k_2$, together with the result above about the pair $(\l_2,\k_2)$, satisfy condition (ii) of Lemma \ref{Lemma7}. For the pair $(\l_2,\k_1)$, the sets $E_1=\O\setminus E_2$, $E_2$, $E_3=E_2+\k_1$, $E_4=E_2+\l_2$ and $E_5=E_2+\l_2+\k_1$ satisfy condition (ii) of Lemma \ref{Lemma6}. For the pair $(\l_1,\k_1)$, the sets $E_1=\O\setminus E_2$, $E_2$, $E_3=E_2+\k_1$, $E_4=E_2+\l_1+\l_2$ and $E_5=E_2+\l_1+\l_2+\k_1$, together with the result above about the pair $(\l_2,\k_1)$, satisfy condition (ii) of Lemma \ref{Lemma7}.

\medskip
IX(d) If $r_2=m_2\not=n_1$, then $m_2>n_1$ and the proof is identical to IX(c) (with $n_2=1$). Let us consider the last case $r_2=n_1=q\ge 2$ with $r_1=1/q$. In this case we will simply choose $\O=\cup_{1\le j\le q}([\frac{j-1}{q},\frac{j}{q})\times [j-1,j))=\cup_{1\le j\le q}([0,\frac{1}{q})\times [0,1)+(j-1)(\k_1+\l_2))$.

\medskip
For the pair $(\l_2,\k_1)$, the sets $E_2=[0,\frac{1}{q})\times [0,1)$, $E_1=\O\setminus \big(E_2\cup (E_2+\k_1+\l_2)\big)$,$E_3=E_2+\k_1$, $E_4=E_2+\l_2$ and $E_5=E_2+\k_1+\l_2$ satisfy condition (i) of Lemma \ref{Lemma6}.

\medskip
For the pair $(\l_1,\k_1)$, the sets $E_2=[0,\frac{1}{q})\times [0,1)$, $E_1=\O\setminus \big(E_2\cup (E_2+\k_1+\l_2)\big)$, $E_3=E_2+\k_1$, $E_4=E_2+\l_1+\l_2$ and $E_5=E_2+\l_1+\l_2+\k_1$, together with the result above about the pair $(\l_2,\k_1)$, satisfy condition (i) of Lemma \ref{Lemma7}.

\medskip
For the pair $(\l_2,\k_2)$, a new approach is needed. Let $E_1=[0,\frac{1}{q})\times [0,1)$, $E_2=E_1+\l_2$, $E_3=E_1+q\l_2=E_1+\k_2$, $E_4=E_1+(q+1)\l_2=E_1+\k_2+\l_2$, $E_1^\p=E_1+\k_1$, $E^\p_2=E^\p_1+\l_2$, $E^\p_3=E^\p_1+q\l_2=E^\p_1+\k_2$, $E^\p_4=E^\p_1+(q+1)\l_2=E^\p_1+\k_2+\l_2$.  Let $\O_0$ be the set $\Omega\setminus \big(E_1\cup E^\p_2\big)$.

\begin{remark}\label{remark1}
{\em
Since $E_1$ and $E^\p_2$ are subsets of $\O$, $E_2$ is both $\L$ and $\K$ disjoint from $\O_0$ since $E_2=E_1+\l_2=E^\p_2-\k_1$. Similarly, one can show that $E_3$, $E_4$, $E_1^\p$, $E_3^\p$ and $E_4^\p$ are all $\L$ and $\K$ disjoint from $\O_0$.
}
\end{remark}

Now define
$$
g(\x)=\frac{\sqrt{d_0}}{2}\big(\chi_{E_1}+\chi_{E_2}+\chi_{E_3}-\chi_{E_4}+\chi_{E^\p_1}+\chi_{E^\p_2}-\chi_{E^\p_3}+\chi_{E^\p_4}\big)+\sqrt{d_0}\chi_{\O_0}.
$$
If $\x\in \O_0$, then $g(\x-\l)=0$ for any $\l\not=\0$ and $g(\x-\k)=0$ for any $\k\not=\0$ by Remark \ref{remark1} above. Thus (\ref{cond12}) and (\ref{cond22}) hold trivially.

\medskip
Now consider the case $\x\in E_1$. Notice that $g(\x-\l)=0$ unless $\l=\0$, $-\l_2$, $-q\l_2$ and $-(q+1)\l_2$, so (\ref{cond12}) holds trivially. On the other hand, for $\l=\0$, $-\l_2$, $-q\l_2$ and $-(q+1)\l_2$ and any $\k\not=\0$, $g(\x-\l-\k)=0$ unless $\k=-\k_1$, $\pm\k_2$ or $-\k_1\pm\k_2$ so (\ref{cond22}) holds trivially unless $\k$ is one of these.

\smallskip
For $\k=-\k_1$:
\begin{eqnarray*}
&&\sum_{\l\in \L}g(\x-\l)\overline{g(\x-\l-\k)}\\
&=& g(\x)\overline{g(\x+\k_1)}+g(\x+\l_2)\overline{g(\x+\l_2+\k_1)}\\
&+&g(\x+q\l_2)\overline{g(\x+q\l_2+\k_1)}+g(\x+(q+1)\l_2)\overline{g(\x+(q+1)\l_2+\k_1)}\\
&=&
\frac{d_0}{4}+\frac{d_0}{4}-\frac{d_0}{4}-\frac{d_0}{4}=0.
\end{eqnarray*}

\smallskip
For $\k=-\k_2$:
\begin{eqnarray*}
&&\sum_{\l\in \L}g(\x-\l)\overline{g(\x-\l-\k)}\\
&=& g(\x)\overline{g(\x+\k_2)}+g(\x+\l_2)\overline{g(\x+\l_2+\k_2)}\\
&=&
\frac{d_0}{4}-\frac{d_0}{4}=0.
\end{eqnarray*}

\smallskip
For $\k=\k_2$:
\begin{eqnarray*}
&&\sum_{\l\in \L}g(\x-\l)\overline{g(\x-\l-\k)}\\
&=& g(\x+(q+1)\l_2)\overline{g(\x+(q+1)\l_2-\k_2)}+g(\x+q\l_2)\overline{g(\x+q\l_2-\k_2)}\\
&=& g(\x+(q+1)\l_2)\overline{g(\x+\l_2)}+g(\x+q\l_2)\overline{g(\x)}\\
&=&
\frac{d_0}{4}-\frac{d_0}{4}=0.
\end{eqnarray*}

\smallskip
For $\k=-\k_1+\k_2$:
\begin{eqnarray*}
&&\sum_{\l\in \L}g(\x-\l)\overline{g(\x-\l-\k)}\\
&=& g(\x+(q+1)\l_2)\overline{g(\x+\k_1+(q+1)\l_2-\k_2)}+g(\x+q\l_2)\overline{g(\x+\k_1+q\l_2-\k_2)}\\
&=& g(\x+(q+1)\l_2)\overline{g(\x+\k_1+\l_2)}+g(\x+q\l_2)\overline{g(\x+\k_1)}\\
&=&
-\frac{d_0}{4}+\frac{d_0}{4}=0.
\end{eqnarray*}

\smallskip
For $\k=-\k_1-\k_2$:
\begin{eqnarray*}
&&\sum_{\l\in \L}g(\x-\l)\overline{g(\x-\l-\k)}\\
&=& g(\x)\overline{g(\x+\k_1+\k_2)}+g(\x+\l_2)\overline{g(\x+\l_2+\k_1+\k_2)}\\
&=&
-\frac{d_0}{4}+\frac{d_0}{4}=0.
\end{eqnarray*}

Thus (\ref{cond22}) holds in general for any $\x\in E_1$. The case of $\x\in E_2$ can be shown in the same way with $-\k_1$ in the above proof replaced by $\k_1$. So $g(\x)$ so defined is a normalized tight Gabor frame function. Let $h(\x)$ be a functional Gabor frame multiplier. Then for $\x\in E_1$ and $\k=-\k_2$, we have
\begin{eqnarray*}
&&\sum_{\l\in \L}(hg)(\x-\l)\overline{(hg)(\x-\l-\k)}\\
&=& h(\x)g(\x)\overline{h(\x+\k_2)g(\x+\k_2)}+h(\x+\l_2)g(\x+\l_2)\overline{h(\x+\l_2+\k_2)g(\x+\l_2+\k_2)}\\
&=&
\frac{d_0}{4}\big(h(\x)\overline{h(\x+\k_2)}-h(\x+\l_2)\overline{h(\x+\l_2+\k_2)}\big)=0.
\end{eqnarray*}
It follows that
$h(\x)\overline{h(\x+\k_2)}=h(\x+\l_2)\overline{h(\x+\l_2+\k_2)}$ and this proves the case for the pair $(\l_2,\k_2)$.

\medskip
Finally, for the pair $(\l_1,\k_2)$, define the sets $E_1$ to $E_4$, $E_1^\p$ to $E_4^\p$ so that $E_1$, $E_3$, $E_1^\p$, $E_3^\p$ are as in the case of the pair $(\l_2,\k_2)$ above, but the other four equal to their corresponding counterparts plus $\l_1$. Define $\O_0=\O\setminus(E_2\cup (E_2+\k_1+\l_2))$ (the same as in the above case). Since $E_2$, $E_4$, $E_2^\p$ and $E_4^\p$ are both $\L$ and $\K$ equivalent to their counterparts as defined in the above case, Remark \ref{remark1} holds for these sets as well. Now the proof for the last case can be copied over with only some minor modifications, and we will leave the details to our reader. This results in $h(\x)\overline{h(\x+\k_2)}=h(\x+\l_1+\l_2)\overline{h(\x+\l_1+\l_2+\k_2)}$, which leads to $h(\x)\overline{h(\x+\k_2)}=h(\x+\l_1)\overline{h(\x+\l_1+\k_2)}$ by the result above for the pair $(\l_2,\k_2)$. This completes the proof for the Type IX case.

\section{The proof of the main result: Part 3}\label{Proof3}

We now give the last part of the proof of Theorem \ref{Theorem1}, that is, the eigen values of $D$ are not real. Without loss of generality we will assume that $D=\begin{bmatrix}a& b\\
-b & a\end{bmatrix}$ with $a^2+b^2\ge 1$ and $b>0$. If $a=0$, then $b\ge 1$ and the statement of Theorem \ref{Theorem1} holds by Remark \ref{Remark7} and the results of the Type II and III cases in Section \ref{Sec_Proof}. Thus, we will only consider the case $a\not=0$. Furthermore, we have the following result as a consequence of the proof of \cite{HW2001} Theorem 1.2, although not in an obvious way (so we will provide a proof).

\begin{lemma}\label{Omega5}
If $a$ and $b$ are not both rational, then the set $\M=\{\l+\k:\ \l\in \L, \k \in \K\}$ is dense in $\R^2$.
\end{lemma}

\begin{proof}
If the statement is not true, that is, $a$ and $b$ are not both rational, and $\M$ is not dense in $\R^2$, then by \cite{HW2001}, there exists a unimodular matrix $P\in M_2(\Z)$ such that the entries in the second row of $PD$ are rational numbers. Let $P=\begin{bmatrix}c_1& c_2\\
c_3 & c_4\end{bmatrix}$ (with $c_j\in \Z$ and $c_1c_4-c_2c_3=1$). This implies that $c_3a-c_4b=r_1$ and $c_3b+c_4a=r_2$ with $r_1$, $r_2\in\Q$. It follows that $a=(c_3r_1+c_4r_2)/(c_3^2+c_4^2)$ and $b=(-c_4r_1+c_3r_2)/(c_3^2+c_4^2)$ are both rational, which is a contradiction.
\end{proof}

Thus there are only two types left for us to consider:

\noindent
Type X: $a$ and $b$ are both rational.\\
\noindent
Type XI: At least one of $a$, $b$ is irrational.

%Type XII: Both $a$, $b$ are irrational, and $a^2+b^2=1$.

\medskip
\textbf{Type X Case}: $a$ and $b$ are both rational so we have $D=\frac{p}{q}D_1$, $D_1=\begin{bmatrix}a_1& b_1\\
-b_1 & a_2\end{bmatrix}$ for some positive integers $p$, $q$, $a_1$, $b_1$ such that $p$, $q$ are comprime and $a_1$, $b_1$ are either both one or are coprime.

\begin{lemma}\label{Omega6}
There exist unimodular matrices $P$, $Q\in M_2(\Z)$ such that $D^\p=PD_1Q=\begin{bmatrix}1& 0\\
0 & z\end{bmatrix}$ where $z=a_1^2+b_1^2\in \Z$. Furthermore, $P$ is a lower triangular matrix with its entries on the main diagonal both being 1.
\end{lemma}

\begin{proof}
If $a_1=b_1=1$, then $P=\begin{bmatrix}1& 0\\
1 & 1\end{bmatrix}$, $Q=\begin{bmatrix}1& -1\\
0 & 1\end{bmatrix}$ and $PD_1Q=\begin{bmatrix}1& 0\\
0 & 2\end{bmatrix}$. If $a_1\not=b_1$, then there exist integers $s$, $t$ such that $sa_1+tb_1=1$ since they are coprime. Let $P=\begin{bmatrix}1& 0\\
sb_1-ta_1 & 1\end{bmatrix}$, $Q=\begin{bmatrix}s& -b_1\\
t & a_1\end{bmatrix}$. We leave it to our reader to verify that $PD_1Q=\begin{bmatrix}1& 0\\
0 & a_1^2+b_1^2\end{bmatrix}$ as we claimed.
\end{proof}

Notice that the statement of Theorem \ref{Theorem1} holds for $\L^\p=\L=\Z^2$, $\K^\p=D^\p\Z^2$ with $D^\p=\frac{p}{q}\begin{bmatrix}1& 0\\
0 & a_1^2+b_1^2\end{bmatrix}=\begin{bmatrix}r_1& 0\\
0 & r_2\end{bmatrix}$ by our proof of the Type IX case in the last section. We also have the following Lemma.

\begin{lemma}\label{Omega7}
If the rectangles $C_j$, $C_j^\p$ in Lemma \ref{newlemma2} are replaced by the parallelograms $P^{-1}C_j$ and $P^{-1}C_j^\p$, then the statement of Lemma \ref{newlemma2} still holds.
\end{lemma}

\begin{proof}
We only need to prove that for each triple $C_j$, $C_j^\p$ and $C_j^\pp$ as given in Lemma \ref{newlemma2}, we have $P^{-1}C_j=P^{-1}C_j^\pp+\l$ for some $\l\in \L$ and $P^{-1}C_j^\p=P^{-1}C_j^\pp+\k$ for some $\k\in \K$. By Lemma \ref{newlemma2}, we have $C_j=C_j^\pp+\l^\p$ for some $\l^\p\in \L$. It follows that $P^{-1}C_j=P^{-1}C_j^\pp+\l$ with $\l=P^{-1}\l^\p\in \L$. Similarly, there exists $\k^\p\in \K^\p=D^\p\Z^2$ (so $\k^\p=D^\p\l^\pp$ for some $\l^\pp\in \Z^2$) such that $C_j=C_j^\pp+\k^\p$. Multiplying both sides by $P^{-1}$, we have $P^{-1}C_j=P^{-1}C_j^\pp+P^{-1}\k^\p$,
and $P^{-1}\k^\p=P^{-1}D^\p\l^\pp=P^{-1}(PDQ)\l^\pp=D(Q\l^\pp)=\k\in \K$ since $Q\l^\pp\in \Z^2$.
\end{proof}

\medskip
\begin{remark}\label{remark5.1} {\em The lower triangular form of $P^{-1}$ means that when it is multiplied to a point in $\R^2$, it does not change the $\l_1$ coordinate of that point. Thus for a rectangle $C$ as described in Lemma \ref{newlemma2}, $P^{-1}C$ is a parallelogram with two vertical sides (parallel to $\l_2$) whose $\l_1$ coordinates are the same as the $\l_1$ coordinates of their corresponding vertical sides in $C$. In particular, $P^{-1}C$ does not change the the vertical side of $C$ if the side is on the $\l_2$ axis.}
\end{remark}

Let $r_1=p/q=m_1/n_1$ (so $m_1=p$ and $n_1=q$), $r_2=p(a_1^2+b_1^2)/q=m_2/n_2$ ($q=cn_2$ with $c=\rm{gcd}(a_1^2+b_1^2,q)$). There are two subcases to consider: 
X(a) $n_1\not=1$ and 
X(b) $n_1=1$.

\medskip
X(a) In this case it is obvious that $\k_1\not\in \L$ and $\k_2\not\in \L$. Let $E_2=P^{-1}\big([0,1/n_1)\times [0,1/n_2)\big)$. Since $P^{-1}S$ tiles $\R^2$ by $\L$, one of the $P^{-1}C_j$'s is $\L$ equivalent to $E_2+\k_1$ (which is not $\L$ equivalent to $E_2$ since $\k_1\not\in \L$). Thus by Lemma \ref{Omega7} we can 
choose  $\O$ so that it contains $E_2$, as well as a parallelogram $E_0$ that is $\L$ equivalent to $E_2+\k_1$ and $\K$ equivalent to $E_2+\l_2$ (which is in the partition of $P^{-1}R$ but not in the partition of $P^{-1}S$ by Remark \ref{remark5.1}). Thus for the pair $(\l_2,\k_1)$, the sets $E_1=\O\setminus\big(E_2\cup E_0)$, $E_3=E_2+\k_1$, $E_4=E_2+\l_2$ and $E_5=E_2+\k_1+\l_2$ satisfy condition (i) of Lemma \ref{Lemma6}. For the pair $(\l_1,\k_1)$, let $E_2=P^{-1}\big([0,1/n_1)\times [0,1/n_2)\big)$ be as defined before. Observe that since $P^{-1}R$ tiles $\R^2$ by $\K$, one of the $P^{-1}C^\p_j$'s in the partition of $P^{-1}R$ is $\K$ equivalent to $E_2+\l_1$. $E_2+\l_1$ cannot be $\K$ equivalent to $E_2$, or we would have $\l_1=j_1\k_1+j_2\k_2=(aj_1-bj_2)\l_1+(bj_1+aj_2)\l_2$ for some $j_1, j_2\in \Z$. It follows that $j_2=-bj_1/a$ and $aj_1-bj_2=aj_1+b^2j_1/a=(a^2+b^2)j_1/a=1$. So $j_1=a/(a^2+b^2)<1$, which is a contradiction. Thus by Lemma \ref{Omega7} we can 
choose  $\O$ so that it contains $E_2$ as well as a parallelogram $E_0$ that is $\L$ equivalent to $E_2+\k_1$ and $\K$ equivalent to $E_2+\l_1$, and the sets $E_1=\O\setminus\big(E_2\cup E_0)$, $E_3=E_2+\k_1$, $E_4=E_2+\l_1$ and $E_5=E_2+\k_1+\l_1$ satisfy condition (i) of Lemma \ref{Lemma6}. The cases for the pairs $(\l_1,\k_2)$ and $(\l_2,\k_2)$ are identical to the above discussions with $\k_1$ replaced by $\k_2$ hence are omitted.

\medskip
X(b) We have $n_1=q=1$ so $r_1=p=m_1$ and $r_2=pz$ with $z=(a_1^2+b_1^2)\ge 2$. In this case $a$ and $b$ are both integers. Similar to the X(a) case, $\l_j\not\in \K$ for $j=1,2$. Notice that in this case we can simply choose $\O$ to be $P^{-1}S$. Consider $E_1=\emptyset$ and $E_2=\O$. Since $\l_j\not\in \K$, $E_2$ is not $\K$ equivalent to $E_2+\l_j$ (by the same argument as in X(a) above), it follows that for any pair $(\l_j,\k_i)$ the sets $E_1=\emptyset$, $E_2=\O$, $E_3=E_2+\k_i$, $E_4=E_2+\l_j$ and $E_5=E_2+\k_i+\l_j$ satisfy condition (ii) of Lemma \ref{Lemma6}.

 \medskip
\textbf{Type XI case:} At least one of $a$, $b$ is irrational.

Let $R$ be the rectangle spanned by $\k_1=D\l_1$ and $\k_2=D\l_2$ (notice that in this case $\k_1$ and $\k_2$ are perpendicular), then $R$ tiles $\R^2$ by $\K$. Without loss of generality let us assume that $R$ lies in the first and fourth quadrants so that $R\cap S\not=\emptyset$ (otherwise we can perform a suitable flip of $S$ around either the $\l_1$ or the $\l_2$ axis). It is necessary that the interior of $R\cap S$ is not empty. By the given condition, there exists $\l^\p\in \L$, $\k^\p\in \K$ such that $\k_j+\l^\p\in S=[0,1)\times [0,1)$ and $\l_i+\k^\p\in R=DS$. It follows that there exists $\delta>0$ small enough such that (i) $E_2=[\delta,2\delta)\times [0,\delta)\subset S$; (ii) $E_2+\l_i+\k^\p\in R$; (iii) $E_2+\k_j+\l^\p\in S$ and (iv) either these three sets are disjoint, or the sets in (ii) and (iii) are identical and are disjoint from $E_2$, in this case it is necessary that $\k_j+\l^\p=\l_i+\k^\p$ (namely $\l^\p=\l_i$ and $\k_j=\k^\p$). By Lemma \ref{newlemma1}, we can then choose $\O$ so that it contains $E_2$ and a measurable set $E_0$ that is $\L$ equivalent to $E_2+\k_j+\l^\p$ and $\K$ equivalent to $
E_2+\l_i+\k^\p$. Since $E_0$ is $\L$ equivalent to $E_2+\k_j$ and $\K$ equivalent to $
E_2+\l_i$, it follows that the sets $E_1=\O\setminus (E_2\cup E_0)$, $E_2$, $E_3=E_2+\k_j$, $E_4=E_2+\l_i$ and $E_5=E_2+\l_i+\k_j$ satisfy condition (i) of Lemma \ref{Lemma6}. This concludes the proof of Theorem \ref{Theorem1}. \qed

\section{Generalizations and Ending Remarks}\label{Sec_Ending}

The same approach used in this paper can be applied to the higher dimensions. For example, Lemma \ref{Lemma2} in fact holds for any dimension $d\ge 2$, hence in order to prove Conjecture \ref{conj1}, it suffices to prove the conjecture with $\L=\Z^d$ and $\K=D\Z^d$ where $D$ is the real Jordan canonical form of $(B^TA)^{-1}$. It is thus quite plausible that Conjecture \ref{conj1} is true in general. For many special cases of $D$, the results and approaches used in this paper can be readily applied. For example $D=\begin{bmatrix}\pm 1& \cdots& 0\\
\vdots&&\vdots\\
0 &\cdots &\pm 1\end{bmatrix}$ (with the proof similar to that of Lemma \ref{Lemma5}(iii)) or $D=\begin{bmatrix}r_1& 0& \cdots& 0\\
0&r_2& \cdots& 0\\
\vdots&&\vdots\\
0 &\cdots &0&r_d\end{bmatrix}$ where the $r_j$'s are either irrational or non-integer rational numbers. Of course, there are many more cases to consider in the general case so the task is harder. It is our intention to tackle the general case in the near future.

\end{document}